\documentclass[11pt]{article}
\usepackage[a4paper,margin=2.5cm]{geometry}
\usepackage[T1]{fontenc}
\usepackage{lmodern}
\usepackage{amsmath,amssymb,amsthm,mathtools}
\usepackage{bbm}
\usepackage{enumitem}
\usepackage{hyperref}
\usepackage{xcolor}
\usepackage{amsfonts,dsfont,pifont}
\usepackage{graphicx}
\usepackage{mathabx}
\usepackage{booktabs}
\usepackage{longtable}
\usepackage{array}
\usepackage[normalem]{ulem}

\title{Genetic weight of a mutant in a biparental Moran model with infinite selection}
\author{
Camille Coron%
\thanks{Université Paris-Saclay, AgroParisTech, INRAE, UMR MIA Paris-Saclay, 91120, Palaiseau, France. Email: \texttt{camille.coron@inrae.fr}}
\and
Olivier Hénard%
\thanks{Université Paris-Saclay, CNRS, Laboratoire de mathématiques d'Orsay, 91405, Orsay, France. Email: \texttt{olivier.henard@universite-paris-saclay.fr}}
}

\date{\today}

\theoremstyle{plain}
\newtheorem{theorem}{Theorem}
\newtheorem{proposition}{Proposition}
\newtheorem{lemma}{Lemma}
\theoremstyle{remark}
\newtheorem{remark}{Remark}

\theoremstyle{remark}

\begin{document}

\maketitle

\begin{abstract}
We study, after Coron Le Jan \cite{coron2024-ukrainian}, the genetic contribution of a single strongly advantageous mutant in a biparental Moran model under infinite selection. We prove a functional limit theorem for the entire beneficial mutation spreading period: after normalization by \(\sqrt N\), the contributions carried by the advantaged and disadvantaged subpopulations converge jointly in law to a non-degenerate random multiple of an explicit deterministic profile. Randomness is generated during the initial microscopic growth of the mutant lineage and subsequently propagated through the population along a deterministic macroscopic trajectory. At the heart of this phenomenon is an eigenmode martingale that measures the imbalance between the mean genetic weight of the founding ancestor within the two classes; its limiting value provides precisely the random amplitude of the macroscopic profile and is characterized by a smoothing equation.
\end{abstract}

\section{Context}\label{sec:introduction}

When a strongly beneficial mutation appears in a single individual, its spread through the population affects not only the selected allele itself but also the transmission of the whole genetic material carried by its original bearer, a phenomenon known as genetic hitchhiking~\cite{smith_haigh_1974}. In this work, we study this effect in a biparental Moran model with infinite viability selection.

The Moran model is a classical population genetics model in which the population size remains constant and reproduction and replacement occur one individual at a time; see~\cite{Moran1958}. In the usual monoparental version, each newborn has a single parent, so that the genetic contribution of a given ancestor (i.e. the quantity of genome of the population coming from this individual) is simply the number of its currently alive descendants. In a biparental population, the situation is fundamentally different. Each newborn has two parents, and in the model considered here, the genetic contribution from any given ancestor carried by this newborn is defined as the average of the contributions carried by those two parents. Contributions are therefore real-valued and evolve through series of averaging along the random pedigree. 

We start from a single individual carrying a strongly beneficial mutation, and whose genetic contribution through time is our object of interest. Selection acts through death: individuals carrying the beneficial mutation are less likely to be replaced. This mutation is transmitted from parents to offspring following a Mendelian rule at one haploid locus (i.e. the type, advantaged or not, of a newborn is chosen uniformly at random among the types of its two parents). By contrast, the genetic contribution carried by an offspring is defined to be the average of the genetic contributions carried by its two parents. This rule may be viewed as an idealized genome-wide approximation, corresponding heuristically to a law-of-large-numbers effect across many independently transmitted neutral loci. Thus selection does not directly modify the transmission of genetic material at reproduction; rather, by changing which individuals survive and are replaced, it reshapes the pedigree along which this material is propagated.

Biparental genetics studies started with the articles \cite{Chang1999} dealing with time scales to common ancestors and \cite{Derrida2000} studying the asymptotic quantity of transmitted genome in the neutral biparental Wright-Fisher model. This topic has also received some more recent interest, through the articles \cite{BartonEtheridge2011,MatsenEvans2008,GravelSteel2015} studying ancestors' genetic contribution dynamics in biparental Wright-Fisher populations, the articles \cite{Lambertetal2021,FoutelRodieretal2026} and \cite{PfaffelhuberWakolbinger2023, adeosun2026} studying the genome structure of sexually reproducing populations, and the articles \cite{WakeleyEtAl2012,WakeleyEtAl2016,DiamantidisEtAl2024,Newman2025} studying the conditioning of gene genealogies by the fixed underlying pedigree. Previous work has also addressed asymptotic genetic contribution for the neutral biparental Moran pedigree~\cite{CoronLeJan2022}, the expected genetic contribution of a single strongly advantaged mutant~\cite{coron2024-ukrainian}, and the deterministic large-population regime in which a macroscopic fraction of the initial population is advantageous~\cite{CoronLeJan2025}. Here we consider the singular regime in which the advantageous type starts from a single individual and selection is infinite. We derive a functional limit theorem describing the evolution of the founder's genetic contribution throughout the advantageous allele spreading, from the appearance of the mutant up to fixation.

A distinctive feature of this regime is that randomness survives in the large-population limit. When a positive proportion of the population is initially advantageous, the system is in a law-of-large-numbers regime and the limiting genetic profile is deterministic~\cite{CoronLeJan2025}. Starting from a single mutant produces a different mechanism. During the first generations of the allelic replacement, the small advantaged population develops randomly and generates a non-degenerate random amplitude. Once the advantaged population has reached a mesoscopic and then macroscopic size, fluctuations become negligible at the relevant scale: the random amplitude created during this initial phase is then transported through the population according to an essentially deterministic profile. The limiting process therefore separates into a random amplitude, generated near the origin of the advantageous allele spreading, and a deterministic macroscopic shape.

This separation between the random seeding phase and the subsequent bulk evolution also guides the proof of our main result, stated in Theorem \ref{thm:main}. The early growth of the advantaged population is analysed through a coupling with a random recursive tree, which identifies the law of the random amplitude appearing in the limit as the solution to a fixed point equation. The macroscopic phase is controlled by a martingale built from the difference between the mean genetic contributions carried by advantageous and disadvantaged individuals. Finally, separate approximation arguments connect the initial regime to the bulk evolution and control the terminal window near fixation. Section \ref{sec:structure-proof} describes in more details the structure and main ingredients of the proof of Theorem \ref{thm:main}.

\section{Model and main result}\label{sec:model&result}
\paragraph{Model}
We consider a discrete-time biparental Moran model with two types,
advantaged and disadvantaged, and with infinite selection at death. The population size \(N\ge1\) is fixed, and we write 
\begin{equation*}
        I_N:=\{1,\ldots,N\}.
\end{equation*}
Each individual has a type, which remains fixed throughout its lifetime. We denote by \(\mathcal Y_0\subset I_N\) the set of initially advantaged individuals and, unless otherwise stated, assume 
$\mathcal Y_0=\{1\}$, corresponding to the arrival of a single advantaged mutant.

At each discrete time \(n\in\mathbb Z_+\), two parents \(\mu_n\) and \(\pi_n\) are sampled independently and uniformly from \(I_N\). Independently, as long as $Y_n<N$, one individual \(\kappa_n\) is sampled uniformly from the set of disadvantaged individuals and is replaced by a child of \(\mu_n\) and \(\pi_n\). 
By convention, the child inherits the type of the mother \(\mu_n\): since the two parents are sampled symmetrically, this is equivalent in law to choosing uniformly which parent transmits the allele at the selected locus. An example of simulation of this pedigree and the advantage status of individuals is given in Figure \ref{fig:Pedigree} for $N=4$ and $6$ time steps.

Thus the number of advantaged individuals increases precisely when \(\mu_n\) is advantaged. Sampling the dying individual only among the disadvantaged class is the infinite viability selection assumption. Once the number of advantaged individuals is equal to $N$ all individuals are equivalent and the model becomes neutral, i.e. the dying individual is chosen uniformly in the population. Asymptotic genetic contribution of ancestors in this neutral framework has been studied in \cite{CoronLeJan2022}.

More formally, let \(\mathcal Y_n\) be the advantaged class at time \(n\), and set
\begin{equation*}
        Y_n:=|\mathcal Y_n|,
        \qquad
        \mathcal D_n:=I_N\setminus\mathcal Y_n,
        \qquad
        D_n:=N-Y_n.
\end{equation*}

We shall use the natural filtration of the reproduction process,
\begin{equation*}
        \mathcal F_n
        =
        \sigma\bigl(\mathcal Y_0,(\mu_\ell,\pi_\ell,\kappa_\ell):0\le \ell<n\bigr).
\end{equation*}
Thus \(\mathcal Y_n,A_n,B_n,C_n\) are \(\mathcal F_n\)-measurable, and,
conditionally on \(\mathcal F_n\), the parents \(\mu_n,\pi_n\) are independent uniform variables on \(I_N\), while \(\kappa_n\) is uniform on \(\mathcal D_n\) (as long as $Y_n<N$, then uniform on $I_N$), independently of the parents.

We encode genetic inheritance through stochastic matrices
\((A_n)_{n\ge0}\), with \(A_n\in[0,1]^{N\times N}\), defined by
\begin{equation*}
        A_0(i,j):=\mathbf 1_{\{i=j\}},
\end{equation*}
and, recursively,
\begin{equation}\label{eq:dyn-A}
        A_{n+1}(i,j)
        =
        \begin{cases}
        A_n(i,j), & i\neq \kappa_n,\\[4pt]
        \dfrac{A_n(\mu_n,j)+A_n(\pi_n,j)}{2}, & i=\kappa_n.
        \end{cases}
\end{equation}
Therefore, conditionally on \(\mathcal F_n\), since the genealogy of a gene at a locus that is independent from the locus under selection can be seen as a random walk on the pedigree backwards in time,
\begin{equation*}
        A_n(i,j)
        =
        \mathbb P\left(
        \text{a gene sampled from individual }i\text{ at time }n
        \text{ descends from }j\text{ at time }0
        \,\middle|\,\mathcal F_n
        \right).
\end{equation*}
In this sense, \(A_n(i,j)\) is the genetic contribution of ancestor \(j\) to the genome of individual \(i\) at time \(n\) and $\sum_{i=1}^NA_n(i,j)$ is the total genetic contribution of ancestor $j$ at time $n$. Figure \ref{fig:Pedigree} shows a simulation of the population with $N=4$ until fixation of the advantaged type as well as the genetic contributions of the initially advantaged ancestor $1$ in each individual once this fixation has occurred.

\begin{figure}\label{fig:Pedigree}
\center
    \includegraphics[scale=0.9]{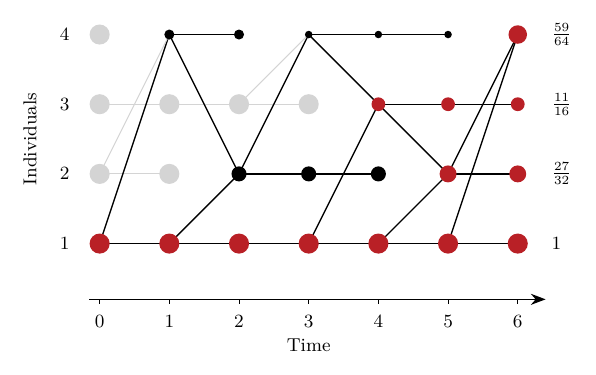}
\caption{A small $N$ instance of the two-type Moran model, starting from a single advantaged individual (in red). Descendants of the initial mutant are shown in black or red according to their type (disadvantaged or advantaged, respectively), which is inherited from the mother. Circle radii are proportional to the fraction of genome inherited from the initial mutant. All other individuals are shown in light grey. The numbers on the right indicate the genetic contribution of the initial mutant (ancestor $1$) to each individual in the final generation, namely $A_{S_N}(1,j)$ for $1\leq j\leq N$.}
\end{figure}

We focus on the total genetic contribution $\sum_{i=1}^NA_n(i,1)$ of the initially advantaged individual. Following \cite{CoronLeJan2025}, we decompose it according to the type of individuals carrying this contribution:
\begin{equation}\label{def-BC}
        B_n
        :=
        \sum_{i\in\mathcal Y_n}A_n(i,1),
        \qquad
        C_n
        :=
        \sum_{i\in\mathcal D_n}A_n(i,1).
\end{equation}
The variable \(B_n\) is the genetic contribution carried by advantaged individuals, while \(C_n\) is the contribution carried by disadvantaged individuals. Hence $B_n+C_n$ is the total genetic contribution of ancestor $1$ at time $n$.\\

Throughout the paper, the dependence in \(N\) is usually suppressed from the notation.

\paragraph{Main result}

For \(1\le k\le N\), define
\[
        S_k:=\inf\{n\ge0:Y_n=k\}.
\]
Thus \(S_1=0\) under our standing assumption \(\mathcal Y_0=\{1\}\). The time \(S_N\) is the fixation time of the advantaged type.
Since \(Y_n\) is nondecreasing and \(\mathcal F_n\)-measurable, \(S_k\) is an
\((\mathcal F_n)\)-stopping time. For \(y\in[0,1]\), set
\[
        k_N(y):=1\vee\lfloor yN\rfloor
\] and equip $\mathbb{R}^2$ with the infinite norm $|.|$.

\begin{theorem}\label{thm:main}
As \(N\to\infty\),
\[
        \left(
        \frac{1}{\sqrt N}
        \binom{B_{S_{k_N(y)}}}{C_{S_{k_N(y)}}}
        \right)_{0\leq y\leq 1}
        \Longrightarrow
        \left(
        W\binom{b(y)}{c(y)}
        \right)_{0\leq y\leq 1}.
\]
in \(D([0,1],\mathbb R^2)\), endowed with the Skorokhod \(J_1\)
topology, where
\[
        b(y):=\frac{2}{\sqrt\pi}\sqrt y\,(1+y),
        \qquad
        c(y):=\frac{2}{\sqrt\pi}\sqrt y\,(1-y),
        \qquad y\in[0,1].
\]
The random variable \(W\) is the unique \(L^2\) solution with mean \(1\)
of
\begin{equation*}
        W
        \stackrel d=
        \sqrt U\,W^{(1)}
        +
        \frac12\sqrt{1-U}\,W^{(2)},
\end{equation*}
where \(U\sim\operatorname{Unif}[0,1]\), and
\(W^{(1)}\) and \(W^{(2)}\) are independent copies of \(W\),
independent of \(U\). Moreover, \(W>0\) almost surely.
\end{theorem}

\begin{remark}
In particular, at fixation,
\[
        \frac{B_{S_N}}{\sqrt N}
        \Longrightarrow
        \frac{4}{\sqrt\pi}W
        \qquad\text{as \(N\to\infty\).}
\]
The corresponding first moment also satisfies
$\mathbb{E}[B_{S_N}]/\sqrt N
        \longrightarrow
        4/\sqrt\pi.$
This follows from the explicit first-moment formula of Lemma~\ref{lem:uv-apriori}, applied at level \(N-1\), together with the fact that the final increment from \(S_{N-1}\) to \(S_N\) is bounded. This convergence of the first moment was previously proved in \cite{coron2024-ukrainian}. Theorem~\ref{thm:main} identifies the full limiting distribution and, more generally, the evolution of the genetic contribution throughout the fixation of the advantageous type.
\end{remark}

\begin{remark}
It is useful to compare Theorem~\ref{thm:main} with the regime of a macroscopic initial condition. Suppose that \(aN\) individuals are initially advantaged, with \(a\in(0,1)\), and let \(B_n\) (resp. $C_n$) denote the total genetic contribution of this initial advantaged population carried by advantaged (resp. disadvantaged) individuals. In this regime the natural scaling is \(N\), and the limit is deterministic.
More precisely, the results of \cite{CoronLeJan2025} suggest (by taking the $s\to \infty$ limit there that:
\[
        \sup_{y\in[a,1]}
        \left|
        \frac1N
        \binom{B_{S_{\lfloor yN\rfloor}}}{C_{S_{\lfloor yN\rfloor}}}
        -
        \binom{b_a(y)}{c_a(y)}
        \right|
        \xrightarrow[N\to\infty]{\mathbb P}
        0,
\]
where
\[
        b_a(y)
        =
        \sqrt{ay}
        +
        y\sqrt a\,(\sqrt y-\sqrt a),
        \qquad
        c_a(y)
        =
        (1-y)\sqrt a\,(\sqrt y-\sqrt a).
\]
Formally substituting \(a=1/N\) into these profiles recovers the \(\sqrt N\)-scale and the shape \(\sqrt y(1\pm y)\), but neither the prefactor \(2/\sqrt\pi\) nor the random multiplier \(W\). These are both created by the microscopic seeding phase. See also Remark \ref{rem:Kurtz-method} for a thorough comparison with \cite{CoronLeJan2025}.
\end{remark}

\section{Proof strategy}\label{sec:structure-proof}

The proof is organized around a triangular drift structure for the process observed at the hitting times \(S_k\), i.e. when the advantaged population first reaches size \(k\). Writing
\begin{equation}\label{eq-X-strategy}
        X_k:=\overline B_k-\overline C_k,
\end{equation}
with $$\overline{B}_k=\frac{B_{S_k}}{k} \qquad \text{and}\qquad\overline{C}_k=\frac{C_{S_k}}{N-k},$$
an exact computation over the block \([S_k,S_{k+1}]\) gives (Proposition \ref{prop:M})
\[
        \mathbb E\left[
        \begin{pmatrix}
        X_{k+1}\\
        \overline C_{k+1}
        \end{pmatrix}
        \,\middle|\,
        \mathcal F_{S_k}
        \right]
        =
        \begin{pmatrix}
        \alpha_{k,N} & 0\\[1mm]
        (2N+1)^{-1} & 1
        \end{pmatrix}
        \begin{pmatrix}
        X_k\\
        \overline C_k
        \end{pmatrix}.
\] Thus the first coordinate evolves
autonomously at the level of conditional drift, while the second is driven
by the first.

After multiplication by an explicit deterministic factor \(a_{k,N}\),
the first coordinate becomes the martingale
\[
        M_k^{(N)}:=a_{k,N}X_k.
\]
At the same time, the second coordinate admits the exact decomposition
\begin{equation}\label{eq-decompo-C-strategy}
        \overline C_{k+1}
        =
        \overline C_k
        +
        \frac{1}{2N+1}a_{k,N}^{-1}M_k^{(N)}
        +
        \Delta R_{k+1},
        \qquad \text{where}\qquad
        \mathbb E\left[
        \Delta R_{k+1}
        \,\middle|\,
        \mathcal F_{S_k}
        \right]
        =0.
\end{equation}
These two identities contain the basic mechanism of the proof: the random
amplitude is carried by \(M^{(N)}\), while the macroscopic profile is obtained
by integrating this mode through the evolution of \(\overline C\).

We choose a mesoscopic sequence \(k_0=k_0(N)\), such that $1<<k_0<<N$ when $N$ goes to infinity. If paternal contributions are discarded, the maternal genealogy of the advantaged individuals is exactly a random recursive tree; see Lemma~\ref{lem:maternal-equals-RRT}. The associated genetic weight, after normalization, is an \(L^2\)-bounded martingale whose limit is denoted \(W\); see Proposition~\ref{prop:RRT-mart}. 
The recursive decomposition of the tree gives a smoothing equation for \(W\), whose \(L^2\)-solution with mean \(1\) is unique and strictly positive; see Proposition~\ref{prop:RRT-smoothing}.
The paternal contribution up to level \(k_0\) is proven to be negligible (Lemma~\ref{lem:paternal-error}), and the disadvantaged contribution is negligible at the same scale by Lemma~\ref{lem:uv-apriori}. Consequently, Proposition~\ref{prop:B-seeding} gives
\[
        M_{k_0}^{(N)}
        \Longrightarrow W.
\]

The second phase propagates this random entrance value through the
macroscopic regime. The one-step quadratic-variation estimate of
Proposition~\ref{prop:QV-local-corr}, together with the a priori bounds of
Lemma~\ref{lem:uv-apriori}, yield (Lemma~\ref{lem:M-flat-away-origin}).
\[
        \sup_{k_0\leq k\leq(1-\varepsilon)N}
        \left|M_k^{(N)}-M_{k_0}^{(N)}\right|
        \xrightarrow[N\to\infty]{L^2}
        0.
\]
Obtaining this estimate is the main
technical effort of the proof. It contrasts with the \(N\)-scale regime of a
macroscopic initial condition studied in~\cite{CoronLeJan2025}, where much
rougher bounds suffice.

Proposition~\ref{prop:Cbar-remainder-variance} shows in parallel that the
accumulated centred remainder in the \(\overline C\)-decomposition is
negligible. Hence, throughout the macroscopic phase,
\[
        M_k^{(N)}\approx W,
        \qquad
        X_k=a_{k,N}^{-1}M_k^{(N)}
        \approx
        \frac{2W}{\sqrt{\pi k}}.
\]
Summing the drift of \(\overline C\) in Equation \eqref{eq-decompo-C-strategy} therefore gives, heuristically,
\[
        \overline C_k
        \approx
        \frac{W}{2N}
        \sum_{j\leq k}a_{j,N}^{-1}
        \approx
        \frac{2W}{\sqrt\pi}\frac{\sqrt{k}}{N}.
\]
Thus, for \(k\sim yN\),
\[
        C_k=(N-k)\overline C_k
        \approx
        \frac{2}{\sqrt\pi}
        W\sqrt N\,\sqrt y\,(1-y),
\]
while using Equation\eqref{eq-X-strategy} gives
\[
        B_k
        =
        k(\overline C_k+X_k)
        \approx
        \frac{2}{\sqrt\pi}
        W\sqrt N\,\sqrt y\,(1+y).
\]
This is precisely the deterministic profile appearing in Theorem~\ref{thm:main};
the rigorous uniform approximation is established in
Proposition~\ref{prop:deterministic-propagation}.

It remains to connect the bulk approximation to the endpoints. Near \(0\),
Lemma~\ref{lem:initial-window} shows that both contributions are negligible on
the \(\sqrt N\)-scale before the mesoscopic entrance level. Near \(1\), the
average disadvantaged contribution is no longer sufficiently controlled to
continue the quadratic-variation argument. A separate terminal estimate,
Lemma~\ref{lem:terminal-window}, shows instead that \(B_k\) varies little and
that \(C_k\) remains small near fixation. Combining these endpoint estimates
with Proposition~\ref{prop:deterministic-propagation} yields directly the
uniform approximation
\[
        \sup_{0\le y\le1}
        \left|
        \frac1{\sqrt N}
        \binom{B_{k_N(y)}}{C_{k_N(y)}}
        -
        M_{k_0}^{(N)}
        \binom{b(y)}{c(y)}
        \right|
        \xrightarrow[N\to\infty]{\mathbb P}0.
\]
The theorem then follows from the convergence of the entrance value
\(M_{k_0}^{(N)}\) and Slutsky's lemma.

\begin{remark}
\label{rem:Kurtz-method}
The macroscopic phase could also be treated as in \cite{CoronLeJan2025},
using a fluid-limit argument in the spirit of Kurtz~\cite{kurtz1970solutions}.
Indeed, setting
\[
    \left(
        \frac{Y_{\lfloor Nt\rfloor}}{N},
        \frac{B_{\lfloor Nt\rfloor}}{\sqrt N},
        \frac{C_{\lfloor Nt\rfloor}}{\sqrt N}
    \right)
    \approx
    \bigl(y(t),b(t),c(t)\bigr),
\]
the conditional drift identities lead to the closed system
\[
\left\{
\begin{aligned}
    \dot y
        &= y,\\
    \dot b
        &= \frac12\bigl((1+y)b+yc\bigr),\\
    \dot c
        &= \frac{1-y}{2}b
           -\frac{y(3-y)}{2(1-y)}c,
\end{aligned}
\right.
\]
A direct fluid approximation from our initial state would, however, be
degenerate: a single initially advantaged individual corresponds in the
limit to the initial condition $(y,b,c)=(0,0,0)$, whose solution remains
identically zero. One could instead start the fluid approximation once
the advantaged population has reached a fixed proportion $a>0$, as in
\cite{CoronLeJan2025}. From that point on, concentration indeed holds and could be proved
by quadratic-variation estimates of the same type as those used below; using $y$ as the time variable, this system then yields
the same deterministic profiles as those we obtain. 
This would still leave the separate problem of gluing the fluid regime
to the microscopic random seeding phase near $y=0$.

We find the martingale $M_k^{(N)}$ more convenient for our purpose,
since the same quantity captures the random amplitude created in the
initial regime and transports it through the macroscopic one.
\end{remark}

\section{Embedded dynamics and the fundamental martingale}
\label{sec:bookkeeping}

All quantities depend on $N$, but this dependence is suppressed unless it is needed.

Generic finite constants are denoted by $K$, or by $K_\varepsilon$ when they may depend on $\varepsilon$, and may change from line to line.

\subsection{One-step and embedded dynamics}
For $Y_n<N$, set
\begin{equation*}
        \overline B_n:=\frac{B_n}{Y_n},
        \qquad
        \overline C_n:=\frac{C_n}{N-Y_n},
        \qquad
        X_n:=\overline B_n-\overline C_n.
\end{equation*}

The quantities $\overline{B}_n$ and $\overline{C}_n$ are respectively equal to the mean genetic contribution of the initial ancestor carried by advantaged (resp. disadvantaged) individuals living at time $n$. The quantity $X_n$ therefore measures a form of imbalance between the proportion of genome of advantaged vs disadvantaged individuals living at time $n$ coming from the initially advantaged individual. This quantity is initially equal to $1$ and due to the even parental genome transmission we expect this difference to tend to go to $0$. The following propositions give first the autonomous dynamics of the number of advantaged individuals $(Y_n)_{n\geq0}$ and second that of $(B_n)_{n\geq0}$, $(C_n)_{n\geq0}$ and the dynamics of the expectation of their rescaled versions $(\overline{B}_n)_{n\geq0}$, $(\overline{C}_n)_{n\geq0}$ and of the difference $(X_n)_{n\geq0}$. The two propositions are proved together for more simplicity.

\begin{proposition}[Autonomous dynamics of $Y$]\label{prop:fund-dynamics-Y}
Let \(w_n(i):=A_n(i,1)\). Conditionally on \(\mathcal F_n\), as long as $Y_n<N$,
\[
        \mathbb P(Y_{n+1}=Y_n+1\mid\mathcal F_n)=\frac{Y_n}{N},
        \qquad
        \mathbb P(Y_{n+1}=Y_n\mid\mathcal F_n)=1-\frac{Y_n}{N}.
\]
Consequently, for \(1\le k\le N-1\),
\[
        S_{k+1}-S_k-1
        \sim \operatorname{Geom}_{\mathbb Z_+}\left(\frac{k}{N}\right)
\] conditionally on \(\mathcal F_{S_k}\).
\end{proposition}

Once $Y_n=N$ the model all individuals are advantaged and the model becomes the neutral version studied in \cite{CoronLeJan2022}. Here we do not consider the dynamics of the considered processes after that time.

\begin{proposition}[One-step dynamics of $B$ and $C$]\label{prop:fund-dynamics-BC}

Let us set $Y_n=k$ and $d=N-Y_n=N-k$. On the stable event
\(\{Y_{n+1}=Y_n=k\}\), one has
\[
        B_{n+1}=B_n,
        \qquad
        C_{n+1}
        =
        C_n+\frac{w_n(\mu_n)+w_n(\pi_n)}2-w_n(\kappa_n),
\]
where, conditionally on \(\mathcal F_n\) and on this event, \(\mu_n\) and
\(\kappa_n\) are uniform on the disadvantaged class, while \(\pi_n\) is uniform
on \(I_N\). In particular,
\begin{equation}\label{eq-Cbar-stable-exp}
        \mathbb E[
        \overline C_{n+1}-\overline C_n
        \mid \mathcal F_n,\,Y_{n+1}=Y_n=k]
        =
        \frac{k}{2Nd}X_n,
\end{equation}
and
\[
        \mathbb E[
        X_{n+1}
        \mid \mathcal F_n,\,Y_{n+1}=Y_n=k]
        =
        \left(1-\frac{k}{2Nd}\right)X_n .
\]

On the jump event \(\{Y_{n+1}=Y_n+1\}\), one has
\[
        B_{n+1}
        =
        B_n+\frac{w_n(\mu_n)+w_n(\pi_n)}2,
        \qquad
        C_{n+1}=C_n-w_n(\kappa_n),
\]
where, conditionally on \(\mathcal F_n\) and on this event, \(\mu_n\) is
uniform on the advantaged class, \(\kappa_n\) is uniform on the disadvantaged
class, and \(\pi_n\) is uniform on \(I_N\). In particular, if \(k\le N-2\), then
\begin{equation}\label{eq-Cbar-jump-exp}
        \mathbb E[
        \overline C_{n+1}-\overline C_n
        \mid \mathcal F_n,\,Y_{n+1}=Y_n+1=k+1]
        =
        0,
\end{equation}
and
\[
        \mathbb E[
        X_{n+1}
        \mid \mathcal F_n,\,Y_{n+1}=Y_n+1=k+1]
        =
        \left(1-\frac{d}{2N(k+1)}\right)X_n .
\]
\end{proposition}

\begin{proof}
The event \(Y_{n+1}=Y_n+1\) occurs exactly when the mother \(\mu_n\) is advantaged, which has conditional probability \(Y_n/N\). Otherwise the number of advantaged individuals is unchanged. This also gives the conditional geometric law of \(S_{k+1}-S_k-1\).

The formulae for \(B_{n+1}\) and \(C_{n+1}\) follow directly from the update rule
for \(A_n\) given in Equation \eqref{eq:dyn-A} and the definitions of $B_n$ and $C_n$ in Equation \eqref{def-BC}.

On a stable step, $Y_{n+1}=k$ and $N-Y_{n+1}=d$, and
\[
        \overline C_{n+1}-\overline C_n
        =
        \frac1d
        \left(
        \frac{w_n(\mu_n)+w_n(\pi_n)}2-w_n(\kappa_n)
        \right).
\]
Taking conditional expectations gives
\[
\begin{aligned}
        \mathbb E[
        \overline C_{n+1}-\overline C_n
        \mid \mathcal F_n,\,Y_{n+1}=Y_n=k]
        &=
        \frac1d
        \left(
        \frac12\overline C_n
        +\frac12\frac{k\overline B_n+d\overline C_n}{N}
        -\overline C_n
        \right)                                      \\
        &=
        \frac{k}{2Nd}
        \left(\overline B_n-\overline C_n\right)
        =
        \frac{k}{2Nd}X_n.
\end{aligned}
\]
Since \(B_{n+1}=B_n\) on a stable step, \(\overline B_{n+1}=\overline B_n\),
and therefore
\[
        \mathbb E[
        X_{n+1}
        \mid \mathcal F_n,\,Y_{n+1}=Y_n=k]
        =
        \left(1-\frac{k}{2Nd}\right)X_n .
\]

On a jump step, assuming \(k\le N-2\), the disadvantaged average satisfies
\[
        \overline C_{n+1}-\overline C_n
        =
        \frac{\overline C_n-w_n(\kappa_n)}{d-1}.
\]
Since \(\kappa_n\) is uniform on the disadvantaged class, its conditional
expectation is zero. For the advantaged average,
\[
        \overline B_{n+1}
        =
        \frac{k}{k+1}\overline B_n
        +
        \frac{w_n(\mu_n)+w_n(\pi_n)}{2(k+1)}.
\]
Thus, taking conditional expectations gives
\[
\begin{aligned}
        \mathbb E[
        \overline B_{n+1}-\overline B_n
        \mid \mathcal F_n,\,Y_{n+1}=Y_n+1=k+1]
        &=
        -\frac1{k+1}\overline B_n
        +
        \frac1{2(k+1)}
        \left(
        \overline B_n
        +
        \frac{k\overline B_n+d\overline C_n}{N}
        \right)                                      \\
        &=
        -\frac{d}{2N(k+1)}
        \left(\overline B_n-\overline C_n\right).
\end{aligned}
\]
Combining this with the zero drift of \(\overline C\) at the jump gives
\[
        \mathbb E[
        X_{n+1}
        \mid \mathcal F_n,\,Y_{n+1}=Y_n+1]
        =
        \left(1-\frac{d}{2N(k+1)}\right)X_n .
\]
\end{proof}

Denote by
\[
        \alpha_{k,N}
        :=
        \frac{2Nk+N+k}{(2N+1)(k+1)}
        =
        \frac{k+1-b_N}{k+1},
        \qquad
        b_N:=\frac{N+1}{2N+1},
\]
and set
\begin{equation}\label{def:akN}
        a_{1,N}:=1,
        \qquad
        a_{k,N}
        :=
        \prod_{\ell=1}^{k-1}\alpha_{\ell,N}^{-1},
        \qquad
        2\leq k\leq N.
\end{equation}

Recall the hitting times \(S_k\) introduced in Section~\ref{sec:model&result}.
From now on, an index \(n\) refers to calendar time, whereas an index \(k\)
refers to the embedded chain observed when the advantaged population has size
\(k\). Thus
\[
        B_k:=B_{S_k},
        \qquad
        C_k:=C_{S_k},
        \qquad
        1\leq k\leq N,
\]
and, for \(1\leq k\leq N-1\),
\[
        \overline B_k:=\frac{B_k}{k},
        \qquad
        \overline C_k:=\frac{C_k}{N-k},
        \qquad
        X_k:=\overline B_k-\overline C_k.
\]
Accordingly, set
\begin{equation}\label{def-M}
        M_k^{(N)}:=a_{k,N}X_k,
        \qquad
        1\leq k\leq N-1.
\end{equation}

Thus \(B_k\) and \(C_k\) always denote the values at the hitting time
\(S_k\), rather than at deterministic calendar time \(k\).

For later use, we also introduce the block notations
\begin{equation}\label{eq:notations}
        d:=N-k,
        \qquad
        q:=\frac dN,
        \qquad
        \lambda:=\frac{k}{2Nd},
        \qquad
        L:=S_{k+1}-S_k-1.
\end{equation}
Thus, \(L\) is the number of stable steps before the jump \(k\to k+1\), and,
conditionally on \(\mathcal F_{S_k}\), from Proposition \ref{prop:fund-dynamics-Y},
\[
        \mathbb P(L=\ell\mid\mathcal F_{S_k})
        =
        q^\ell(1-q),
        \qquad
        \ell\geq0.
\]
On \(\{L\geq j\}\), let \(\overline C^{(j)}\) and \(X^{(j)}\) denote the
values of \(\overline C\) and \(X\) after the first \(j\) stable steps of the block (i.e. at time $S_k+j$), and set
\[
        C^{(j)}:=d\,\overline C^{(j)}.
\]
In particular,
\[
        \overline C^{(0)}=\overline C_k,
        \qquad
        X^{(0)}=X_k,
        \qquad
        C^{(0)}=C_k.
\]

\begin{proposition}[Embedded drift structure and fundamental martingale]
\label{prop:M}
For every \(N\geq3\) and \(1\leq k\leq N-2\),
\begin{equation}
\label{eq:embedded-drift}
        \mathbb E\left[
        \begin{pmatrix}
        X_{k+1}\\
        \overline C_{k+1}
        \end{pmatrix}
        \,\middle|\,
        \mathcal F_{S_k}
        \right]
        =
        \begin{pmatrix}
        \alpha_{k,N} & 0\\[1mm]
        (2N+1)^{-1} & 1
        \end{pmatrix}
        \begin{pmatrix}
        X_k\\
        \overline C_k
        \end{pmatrix}.
\end{equation}
Consequently, $\left(M_k^{(N)}\right)_{1\leq k\leq N-1}$ is a martingale with respect to \((\mathcal F_{S_k})_{1\leq k\leq N-1}\).

Moreover, setting
\[
        \Delta R_{k+1}
        :=
        \overline C_{k+1}
        -
        \overline C_k
        -
        \frac{1}{2N+1}X_k,
\]
one has the decomposition
\begin{equation}
\label{eq:Cbar-drift-decomposition}
        \overline C_{k+1}
        =
        \overline C_k
        +
        \frac{1}{2N+1}X_k
        +
        \Delta R_{k+1}
        =
        \overline C_k
        +
        \frac{1}{2N+1}a_{k,N}^{-1}M_k^{(N)}
        +
        \Delta R_{k+1},
\end{equation}
where
\[
        \mathbb E\left[
        \Delta R_{k+1}
        \,\middle|\,
        \mathcal F_{S_k}
        \right]
        =
        0.
\]
\end{proposition}

\begin{proof}
Fix \(1\leq k\leq N-2\). By Proposition~\ref{prop:fund-dynamics-Y},
conditionally on \(\mathcal F_{S_k}\),
\[
        \mathbb P(L=\ell\mid\mathcal F_{S_k})
        =
        q^\ell(1-q),
        \qquad
        \ell\geq0.
\]
By Proposition~\ref{prop:fund-dynamics-BC}, during each stable step,  \[
        \mathbb E[
        X_{n+1}
        \mid \mathcal F_n,\,Y_{n+1}=Y_n=k]
        =
        \left(1-\lambda\right)X_n,
\] while at the final jump \(k\to k+1\), \[
        \mathbb E[
        X_{n+1}
        \mid \mathcal F_n,\,Y_{n+1}=Y_n+1=k+1]
        =
        \left(1-\frac{d}{2N(k+1)}\right)X_n .
\]
Therefore
\begin{equation}\label{eq-rec-X}
\begin{aligned}
        \mathbb E\left[
        X_{k+1}
        \,\middle|\,
        \mathcal F_{S_k}
        \right]
        &=
        \sum_{\ell\geq0}
        q^\ell(1-q)
        (1-\lambda)^\ell
        \left(
        1-\frac{d}{2N(k+1)}
        \right)
        X_k                                      \\
        &=
        \alpha_{k,N}X_k.
\end{aligned}
\end{equation}

We next compute the drift of \(\overline C\). Since the event that the next update is stable is independent of the current genetic configuration, Proposition~\ref{prop:fund-dynamics-BC} gives, by induction,
\[
        \mathbb E\left[
        X^{(j)}
        \,\middle|\,
        \mathcal F_{S_k},\,L\geq j
        \right]
        =
        (1-\lambda)^jX_k.
\]
The final jump has zero conditional drift for \(\overline C\), therefore
\begin{equation}\label{eq-exp-DeltaC}
\begin{aligned}
        \mathbb E\left[
        \overline C_{k+1}-\overline C_k
        \,\middle|\,
        \mathcal F_{S_k}
        \right]
        &=\mathbb E\left[
        \overline C^{(L)}-\overline C_{k}
        \,\middle|\,
        \mathcal F_{S_k}
        \right]\\&=-\mathbb E\left[
        X^{(L)}-X_{k}
        \,\middle|\,
        \mathcal F_{S_k}
        \right]\\&=-
        \left[\left(\sum_{j\geq0}
        q^{j}(1-q)(1-\lambda)^j\right)-1\right]X_k\\
        &=
        \frac{q\lambda}{1-q(1-\lambda)}X_k\\
        &=
        \frac{1}{2N+1}X_k.
\end{aligned}
\end{equation}
Together with the preceding identity \eqref{eq-rec-X} for \(X_{k+1}\), this proves
\eqref{eq:embedded-drift}.

Finally, since
\[
        a_{k+1,N}
        =
        a_{k,N}\alpha_{k,N}^{-1},
\]
we obtain
\[
        \mathbb E\left[
        M_{k+1}^{(N)}
        \,\middle|\,
        \mathcal F_{S_k}
        \right]
        =
        a_{k+1,N}\alpha_{k,N}X_k
        =
        M_k^{(N)}
\]
therefore $(M_k)_{1\leq k\leq N-1}$ is a martingale with respect \((\mathcal F_{S_k})_{1\leq k\leq N-1}\). The centering of
\(\Delta R_{k+1}\) follows directly from \eqref{eq-exp-DeltaC} and the second part of the decomposition \eqref{eq:Cbar-drift-decomposition} follows from the definition \eqref{def-M}.
\end{proof}

The triangular form of \eqref{eq:embedded-drift} is the basic structural feature used throughout the proof: the conditional expectation of the imbalance \(X_k\) evolves autonomously after deterministic renormalization, while the dynamics of the disadvantaged average \(\overline C_k\) is driven by this imbalance \(X_k\).

\begin{remark}
The martingale $(M_k)_{1\leq k\leq N-1}$ is defined only up to \(k=N-1\), since \(\overline C_N\) is not defined at fixation. The terminal step
\(N-1\to N\) will be handled separately in Section~\ref{sec-gluing}.
\end{remark}

\subsection{Deterministic and first-moment estimates}

We first collect the estimates on the deterministic coefficients introduced
in \eqref{def:akN}.

\begin{lemma}[Deterministic estimates for \(a_{k,N}\)]
\label{lem:a-estimates}
There exists a constant \(K<\infty\) such that, for every \(N\ge2\)
and every \(1\le k\le N\),
\begin{equation}\label{2-sided-estimate-a}
        K^{-1}\sqrt{k}
        \le
        a_{k,N}
        \le
        K\sqrt{k}.
\end{equation}

Moreover, if \(k_0=k_0(N)\leq N\) satisfies \(k_0\to\infty\) when $N$ goes to infinity, then
\begin{equation}\label{uniform-asymptotics-a}
        \sup_{k_0\le k\le N}
        \left|
        \sqrt{k}\,a_{k,N}^{-1}
        -
        \frac{2}{\sqrt{\pi}}
        \right|
        \xrightarrow[N\to\infty]{}0.
\end{equation}
If in addition \(k_0/N\to0\) when $N$ goes to infinity, then
\begin{equation}\label{eq:sum_a-1}
        \sqrt{N}\sup_{k_0\le k\le N}
        \left|\frac{1}{2N+1}
        \sum_{j=k_0}^{k-1}a_{j,N}^{-1}-\frac{2}{\sqrt{\pi}}\frac{\sqrt{k}}{N}\right|\longrightarrow 0 \quad\quad\text{as $N\rightarrow\infty$.}
\end{equation}
\end{lemma}

\begin{proof}
From \eqref{def:akN}, \[ a_{k,N} = \Gamma(2-b_N) \frac{\Gamma(k+1)}{\Gamma(k+1-b_N)}, \qquad b_N=\frac12+\frac{1}{4N+2}. \]
By Gautschi's inequality \cite{Gautschi1959}, for every
\(k\ge1\) and \(0<b<1\),
\[
1\le \frac{\Gamma(k+1)}{k^b\Gamma(k+1-b)}
\le (1+k^{-1})^b\le 1+k^{-1}.
\]
Consequently, since \(b_N\in[1/2,3/5]\),
\[
        \frac{\Gamma(k+1)}{\Gamma(k+1-b_N)}
        =
        k^{b_N}\bigl(1+O(k^{-1})\bigr),
\]
where the implicit constant is independent of \(N\). In particular, this
ratio is bounded above and below by positive constant multiples of
\(k^{b_N}\), uniformly in \(N\ge2\) and \(k\ge1\). Moreover,
\(\Gamma(2-b_N)\) is bounded above and away from zero, while
\[
        k^{b_N}
        =
        \sqrt{k}\exp\left(\frac{\log k}{4N+2}\right)
\]
and, for \(k\le N\),
\[
        1
        \le
        \exp\left(\frac{\log k}{4N+2}\right)
        \le
        \exp\left(\frac{\log N}{4N+2}\right).
\]
The last expression is bounded uniformly in \(N\ge2\), and
\eqref{2-sided-estimate-a} follows.

Moreover,
\[
        \sqrt{k}\,a_{k,N}^{-1}
        =
        \frac{k^{1/2-b_N}}{\Gamma(2-b_N)}
        \bigl(1+O(k^{-1})\bigr).
\]
Uniformly for \(k_0\le k\le N\),
\[
        \left|\log k^{1/2-b_N}\right|
        =
        \frac{\log k}{4N+2}
        \le
        \frac{\log N}{4N+2}
        \longrightarrow0,
\]
and the \(O(k^{-1})\) term is \(o(1)\) since \(k_0\to\infty\). Finally,
\[
        \Gamma(2-b_N)
        \longrightarrow
        \Gamma\left(\frac32\right)
        =
        \frac{\sqrt{\pi}}{2}.
\]
This proves \eqref{uniform-asymptotics-a}.

Set
\[
        \varepsilon_N
        :=
        \sup_{k_0\le j\le N}
        \left|
        \sqrt{j}\,a_{j,N}^{-1}
        -
        \frac{2}{\sqrt{\pi}}
        \right|,
\]
so that \(\varepsilon_N\to0\). Uniformly for \(k_0\le k\le N\),
\[
        \left|
        \sum_{j=k_0}^{k-1}a_{j,N}^{-1}
        -
        \frac{2}{\sqrt{\pi}}
        \sum_{j=k_0}^{k-1}j^{-1/2}
        \right|
        \le
        \varepsilon_N
        \sum_{j=k_0}^{k-1}j^{-1/2}
        \le
        2\varepsilon_N\sqrt{N}.
\]
Furthermore, by comparison with the corresponding integral,
\[
        \sum_{j=k_0}^{k-1}j^{-1/2}
        =
        2\bigl(\sqrt{k}-\sqrt{k_0}\bigr)
        +O(k_0^{-1/2})
\]
uniformly in \(k\ge k_0\). Consequently,
\[
\begin{aligned}
        \sqrt{N}\sup_{k_0\le k\le N}
        \left|
        \frac{1}{2N+1}
        \sum_{j=k_0}^{k-1}a_{j,N}^{-1}
        -
        \frac{2}{\sqrt{\pi}}\frac{\sqrt{k}}{N}
        \right|
        \le
        C\varepsilon_N
        +C\sqrt{\frac{k_0}{N}}
        +\frac{C}{N}
        +\frac{C}{\sqrt{Nk_0}},
\end{aligned}
\]
which tends to zero because \(k_0/N\to0\). This proves
\eqref{eq:sum_a-1}.
\end{proof}

\begin{lemma}[First-moment calculations and bounds]
\label{lem:uv-apriori}
There exists a constant \(K<\infty\) such that, for every
\(1\le k\le N-1\),
\begin{equation}\label{eq-esp-X-control}
        \mathbb E[X_k]
        =
        a_{k,N}^{-1}
        \le
        \frac{K}{\sqrt{k}},
\end{equation}
while
\[
        \mathbb E[\overline C_k]=\frac{ka_{k,N}^{-1}-1}{N}
        \qquad\text{and}\qquad
        \qquad
        \mathbb E[\overline B_k]=\frac{(N+k)a_{k,N}^{-1}-1}{N}.
\] In particular,
\[
        \mathbb E[\overline C_k]\le K\frac{\sqrt{k}}{N},
        \qquad
        \mathbb E[\overline B_k+\overline C_k]
        \le
        \frac{K}{\sqrt{k}},
        \qquad\text{and}\qquad
        \mathbb E[B_k+C_k]
        \le
        K\sqrt{k}.
\]
\end{lemma}

\begin{proof}
By Proposition~\ref{prop:M},
\[
        \mathbb E[M_k^{(N)}]
        =
        \mathbb E[M_1^{(N)}]
        =
        1.
\]
Since \(M_k^{(N)}=a_{k,N}X_k\),
\[
        \mathbb E[X_k]
        =
        a_{k,N}^{-1}.
\]
The inequality given in \eqref{eq-esp-X-control} therefore follows directly from Lemma~\ref{lem:a-estimates}.

Since \(\overline C_1=0\), taking expectations in
\eqref{eq:Cbar-drift-decomposition} and summing gives
\[
        \mathbb E[\overline C_k]
        =
        \frac{1}{2N+1}
        \sum_{j=1}^{k-1}a_{j,N}^{-1}.
\]
What is more by definition, $$a_{k+1,N}^{-1}=\frac{k+1-b_N}{k+1}a_{k,N}^{-1}\qquad \text{which gives} \qquad\frac{a_{k,N}^{-1}}{2N+1}=\frac{k+1}{N}a_{k+1,N}^{-1}-\frac{k}{N}a_{k,N}^{-1}.$$
Hence
$$\mathbb E[\overline C_k]
    =\frac{ka_{k,N}^{-1}-1}{N}
    \le
    K\frac{\sqrt{k}}{N}.$$
Next since $\overline B_k=X_k+\overline C_k,$ $$\mathbb E[\overline B_k]=\frac{(N+k)a_{k,N}^{-1}-1}{N}.$$
and we obtain using \eqref{eq-esp-X-control} that
\[
        \mathbb E[\overline B_k+\overline C_k]
        \le
        \frac{K}{\sqrt{k}}.
\]
Finally, using the non-negativity of the class averages $\overline{B}_k$ and $\overline{C}_k$,
\[
\begin{aligned}
        \mathbb E[B_k+C_k]
        &=
        k\,\mathbb E[\overline B_k]
        +(N-k)\mathbb E[\overline C_k] \\
        &\le
        K\sqrt{k}
        +K\sqrt{k},
\end{aligned}
\]
which proves the last inequality of Lemma \ref{lem:uv-apriori} after changing \(K\).
\end{proof}

\section{Initial window and random recursive tree coupling}\label{sec:initial-window}

The random amplitude is generated while the advantaged population is still small. The key observation is that, if paternal genetic input is discarded, each new advantaged individual inherits half the contribution of its advantaged mother, while that mother is uniform among the current advantaged individuals. The resulting maternal genealogy is therefore a random recursive tree. We first analyse the corresponding weighted tree, and then show that the paternal contribution is negligible on the mesoscopic scale i.e. as long as $k\ll N$.

\subsection{Random recursive tree (RRT)}

Let \((\mathcal T_k)_{k\ge1}\) be the random recursive tree rooted at \(1\): for each
\(k\ge2\), vertex \(k\) is attached to a parent \(U_k\) chosen uniformly in
\(\{1,\ldots,k-1\}\), independently conditionally on the past. We denote by
\(\operatorname{depth}_k(v)\) the distance from \(v\) to the root in \(\mathcal T_k\), and set
\begin{equation}\label{eq:RRT-weight}
        B_k^\star
        :=
        \sum_{v\in\mathcal T_k}2^{-\operatorname{depth}_k(v)},
        \qquad
        \gamma_k
        :=
        \prod_{\ell=1}^{k-1}
        \left(1+\frac1{2\ell}\right),
        \qquad
        M_k^\star:=\frac{B_k^\star}{\gamma_k},
\end{equation}
where the empty product is understood to be equal to \(1\), so that \(\gamma_1=1\).

\begin{remark}
    This martingale is a particular instance of martingales studied by Drmota and Hwang~\cite[Section~3.3]{DrmotaHwang2005}, see also Fuchs, Hwang and Neininger~\cite[Equation~(4) and Proposition~1]{FuchsHwangNeininger2006} for the characterisation of their limit. We nevertheless retain our original derivation below, both to keep the paper self-contained and to establish the properties in the precise form needed here.
\end{remark}

\begin{proposition}[RRT martingale, $L^2$ boundedness and limit]\label{prop:RRT-mart}
The process \((M_k^\star)_{k\ge1}\) is a nonnegative martingale, and
\(\mathbb E[M_k^\star]=1\) for every \(k\ge1\). What is more, \[
\sup_{k\ge 1}\mathbb{E}\bigl[(M_k^\star)^2\bigr]<\infty.
\]
Consequently \(M_k^\star\to M_\infty^\star\) almost surely and in \(L^2\), and $\mathbb{E}[M_\infty^\star]=1.$
\end{proposition}

We first prove the following lemma giving the asymptotics of $\gamma_k$.

\begin{lemma}[Asymptotics of \(\gamma_k\)]\label{lem:gamma-asymptotics}
As \(k\to\infty\),
\[
        \frac{\gamma_k}{\sqrt k}
        \longrightarrow
        \frac{2}{\sqrt\pi}.
\]
\end{lemma}

\begin{proof}
Using the Gamma function,
\[
        \gamma_k
        =
        \prod_{\ell=1}^{k-1}
        \left(1+\frac1{2\ell}\right)
        =
        \frac{\Gamma(k+\frac12)}
        {\Gamma(\frac32)\Gamma(k)}
        =
        \frac{2}{\sqrt\pi}
        \frac{\Gamma(k+\frac12)}{\Gamma(k)}.
\]
The result follows from Stirling's formula.
\end{proof}

\begin{proof}[Proof of Proposition~\ref{prop:RRT-mart}]
Let \(\mathcal H_k:=\sigma(\mathcal T_k)\), and let \(U_{k+1}\) be the parent of vertex
\(k+1\). Then
\[
        B_{k+1}^\star=B_k^\star+\xi_{k+1},
        \qquad
        \xi_{k+1}:=2^{-(\operatorname{depth}_k(U_{k+1})+1)}.
\]
Conditionally on \(\mathcal H_k\), \(U_{k+1}\) is uniform on \(\mathcal T_k\), hence
\[
        \mathbb E[\xi_{k+1}\mid\mathcal H_k]
        =
        \frac1k\sum_{v\in\mathcal T_k}
        2^{-(\operatorname{depth}_k(v)+1)}
        =
        \frac{1}{2k}B_k^\star.
\]
Therefore
\[
        \mathbb E[B_{k+1}^\star\mid\mathcal H_k]
        =
        \left(1+\frac1{2k}\right)B_k^\star,
\]
and the definition of \(\gamma_k\) gives
\[
        \mathbb E[M_{k+1}^\star\mid\mathcal H_k]=M_k^\star.
\]
Since \(M_1^\star=1\), the first sentence of Proposition~\ref{prop:RRT-mart} is proved. Now let
\[
R_k^\star:=\sum_{v\in\mathcal T_k}4^{-\mathrm{depth}_k(v)}.
\]
Then
\[
\mathbb{E}[\xi_{k+1}^2\mid\mathcal H_k]
=
\frac{1}{4k}R_k^\star,
\]
therefore
\[
\mathbb{E}\bigl[(B_{k+1}^\star)^2\mid\mathcal H_k\bigr]
=
(B_k^\star)^2\Bigl(1+\frac1k\Bigr)
+
\frac{1}{4k}R_k^\star.
\]
Since
\[
\gamma_{k+1}=\gamma_k\Bigl(1+\frac{1}{2k}\Bigr),
\]
we get
\[
\mathbb{E}\bigl[(M_{k+1}^\star)^2\mid\mathcal H_k\bigr]
=
(M_k^\star)^2
\frac{1+\frac1k}{(1+\frac{1}{2k})^2}
+
\frac{1}{4k}
\frac{R_k^\star}{\gamma_k^2(1+\frac{1}{2k})^2}.
\]
Now
\[
1+\frac1k\le \Bigl(1+\frac{1}{2k}\Bigr)^2,
\qquad \text{ and } \qquad
R_k^\star\le B_k^\star.
\]
Thus, after taking expectations,
\[
\mathbb{E}\bigl[(M_{k+1}^\star)^2\bigr]
\le
\mathbb{E}\bigl[(M_k^\star)^2\bigr]
+
\frac{1}{4k}\,
\mathbb{E}\left[\frac{B_k^\star}{\gamma_k^2}\right].
\]
Since \(B_k^\star=\gamma_k M_k^\star\) and
\(\mathbb{E}[M_k^\star]=1\),
\[
\mathbb{E}\bigl[(M_{k+1}^\star)^2\bigr]
\le
\mathbb{E}\bigl[(M_k^\star)^2\bigr]
+
\frac{1}{4k\gamma_k}.
\]
By Lemma \ref{lem:gamma-asymptotics}, there exists a constant \(K<\infty\) such that
\[
\frac1{\gamma_k}\le \frac{K}{\sqrt{k}},
\qquad \text{for every}\qquad k\ge1.
\]
Consequently
\[
\mathbb{E}\bigl[(M_{k+1}^\star)^2\bigr]
\le
\mathbb{E}\bigl[(M_k^\star)^2\bigr]
+
\frac{K}{k^{3/2}}.
\]
Iterating gives
\[
\sup_{n\ge1}\mathbb{E}\bigl[(M_n^\star)^2\bigr]
\le
\mathbb{E}\bigl[(M_1^\star)^2\bigr]
+
K\sum_{k\ge1}\frac1{k^{3/2}}
<\infty.
\]
Since \((M_k^\star)_{k\ge1}\) is an \(L^2\)-bounded martingale, it converges almost surely and in \(L^2\) to a limit \(M_\infty^\star\). Finally, \(L^1\)-convergence gives
\[
\mathbb{E}[M_\infty^\star]
=
\lim_{k\to\infty}\mathbb{E}[M_k^\star]
=
1.
\]
\end{proof}

\begin{proposition}[Smoothing equation for the RRT limit]
\label{prop:RRT-smoothing}
The limiting random variable $M^{\star}_{\infty}$ is the unique $L^2$ solution with mean $1$ of the equation in law
\begin{equation}
\label{eq:W-smoothing}
        W
        \stackrel d=
        \sqrt U\,W^{(1)}
        +
        \frac12\sqrt{1-U}\,W^{(2)},
\end{equation}
where \(U\sim\operatorname{Unif}[0,1]\), and
\(W^{(1)},W^{(2)}\) are independent copies of \(W\), independent of \(U\).
What is more, this solution satisfies $M^{\star}_{\infty}>0$ almost surely.
\end{proposition}

\begin{proof}
Remove the edge between vertices \(1\) and \(2\) in the recursive tree
\(\mathcal T_n\). This gives two rooted subtrees: the component containing
vertex \(1\), denoted by \(\mathcal T_n^{(1)}\), and the component rooted at
vertex \(2\), denoted by \(\mathcal T_n^{(2)}\). Let
\[
        K_n:=|\mathcal T_n^{(1)}|.
\]
By the standard Pólya-urn description of the two branches born from the first split (\cite{Mahmoud2008}),
\[
        K_n\sim\operatorname{Unif}\{1,\ldots,n-1\}.
\]
Moreover, conditionally on \(K_n=k\), the two rooted trees
\(\mathcal T_n^{(1)}\) and \(\mathcal T_n^{(2)}\), after relabelling their
vertices in increasing order, are independent recursive trees with respective
sizes \(k\) and \(n-k\).

The component containing vertex \(1\) keeps the same depths as in
\(\mathcal T_n\), whereas every vertex in the component rooted at vertex \(2\)
has its depth increased by one. Hence
\[
        B_n^\star
        \stackrel d=
        B_{K_n}^{\star,(1)}
        +
        \frac12 B_{n-K_n}^{\star,(2)},
\] where $B^{\star,(1)}$ and $B^{\star,(2)}$ are independent processes with same law as $B^{\star}$.
Therefore 
\begin{equation}
\label{eq:RDE}
        M_n^\star
        \stackrel d=
        \frac{\gamma_{K_n}}{\gamma_n}
        M_{K_n}^{\star,(1)}
        +
        \frac12
        \frac{\gamma_{n-K_n}}{\gamma_n}
        M_{n-K_n}^{\star,(2)},
\end{equation}
where the two RRT processes on the right-hand side are independent and
independent of \(K_n\).

We now realize the right-hand side of \eqref{eq:RDE} on a common probability space as follows.
Let \(U\sim\operatorname{Unif}[0,1]\), independent of the two RRT processes, and set
\[
    \widetilde K_n := 1+\lfloor (n-1)U\rfloor .
\]
For every \(n\), \(\widetilde K_n\) has the same law as \(K_n\), so that the right-hand side of
\eqref{eq:RDE} is unchanged in distribution when \(K_n\) is replaced by \(\widetilde K_n\).
Moreover,
\[
    \frac{\widetilde K_n}{n}\longrightarrow U
    \qquad\text{almost surely}.
\]
Since \(U\in(0,1)\) almost surely, both \(\widetilde K_n\) and \(n-\widetilde K_n\)
tend to infinity almost surely. Proposition~\ref{prop:RRT-mart}, applied to the two
independent RRT processes, gives
\[
    M^{\star,(1)}_m \longrightarrow M_{\infty}^{*,1},
    \qquad
    M^{\star,(2)}_m \longrightarrow M_{\infty}^{*,2}
    \qquad\text{almost surely as }m\to\infty,
\]
where $M_{\infty}^{*,1}$ and $M_{\infty}^{*,2}$ are independent copies of $M_{\infty}^{*}$, independent of \(U\).
Hence, along the random indices \(\widetilde K_n\) and \(n-\widetilde K_n\), and using
Lemma~\ref{lem:gamma-asymptotics},
\[
        \left(
        \frac{\gamma_{\widetilde K_n}}{\gamma_n},
        \frac{\gamma_{n-\widetilde K_n}}{\gamma_n},
        M_{\widetilde K_n}^{\star,(1)},
        M_{n-\widetilde K_n}^{\star,(2)}
        \right)
        \longrightarrow
        \left(
        \sqrt U,\sqrt{1-U},M_{\infty}^{*,1},M_{\infty}^{*,2}
        \right)
\]
almost surely. Therefore the right-hand side of \eqref{eq:RDE}, realized with
\(\widetilde K_n\) in place of \(K_n\), converges almost surely to
\[
        \sqrt U\,M_{\infty}^{*,1}
        +\frac12\sqrt{1-U}\,M_{\infty}^{*,2}.
\]
Since this realization has the same distribution as the right-hand side of
\eqref{eq:RDE}, while \(M_n^\star\Rightarrow M_{\infty}^{*}\) by Proposition~\ref{prop:RRT-mart}, we obtain that $M^*_{\infty}$ is solution to \eqref{eq:W-smoothing}.

We now prove uniqueness using a contraction argument. Let \(\Phi\) be the map on the set \(\mathcal P_2(\mathbb R)\) of probability measures on $\mathbb{R}$ with finite second moment defined by
\[
        \Phi(\mu)
        :=
        \mathcal L\left(
        \sqrt U\,X_1
        +
        \frac12\sqrt{1-U}\,X_2
        \right),
\]
where \(X_1,X_2\) are independent with law \(\mu\), independently of \(U\).
If \(\mu,\nu\in\mathcal P_2(\mathbb R)\) both have mean \(1\), take any
coupling \((X,Y)\) of \(\mu,\nu\), and two independent copies
\((X_1,Y_1),(X_2,Y_2)\), independent of \(U\). Since
\(\mathbb E[X-Y]=0\),
\[
\begin{aligned}
        &\mathbb E\left[
        \left(
        \sqrt U(X_1-Y_1)
        +
        \frac12\sqrt{1-U}(X_2-Y_2)
        \right)^2
        \right]
        \\
        &\qquad=
        \left(
        \mathbb E[U]
        +
        \frac14\mathbb E[1-U]
        \right)
        \mathbb E[(X-Y)^2]
        =
        \frac58\mathbb E[(X-Y)^2].
\end{aligned}
\]
Taking the infimum over all couplings gives
\[
        W_2^2(\Phi(\mu),\Phi(\nu))
        \le
        \frac58 W_2^2(\mu,\nu)
\] where $W_2$ denotes the $L^2$-Wasserstein distance on $\mathcal{P}_2(\mathbb{R})$. Thus \(\Phi\) has at most one fixed point in \(\mathcal P_2(\mathbb R)\) with mean \(1\). Since Proposition~\ref{prop:RRT-mart} gives that \(M_{\infty}^*\in L^2\) and \(\mathbb E[M_{\infty}^*]=1\), the first sentence of Proposition \ref{prop:RRT-smoothing} is proved.

Finally, \(M_{\infty}^*\ge0\) almost surely. Setting $p:=\mathbb P(M_{\infty}^*=0)$, the almost sure strict positivity of both coefficients $\sqrt{U}$ and $\sqrt{1-U}$ in \eqref{eq:W-smoothing} gives
\[p=p^2.\]
Since \(M_{\infty}^*\) is not identically zero (since $\mathbb{E}[M_{\infty}^*]=1$), \(p=0\), and hence \(M_{\infty}^*>0\) almost surely.
\end{proof}

\begin{remark}

Taking second moments in \eqref{eq:W-smoothing} together with the condition that $\mathbb{E}[W]=1$ gives
\[
        \mathbb E[W^2]
        =
        \left(
        \mathbb E[U]
        +
        \frac14\mathbb E[1-U]
        \right)
        \mathbb E[W^2]
        +
        \mathbb E\left[\sqrt{U(1-U)}\right].
\]
Since
\[
        \mathbb E[U]=\frac12,
        \qquad
        \mathbb E\left[\sqrt{U(1-U)}\right]=\frac{\pi}{8},
\]
we obtain
\[
        \mathbb E[W^2]=\frac{\pi}{3},
        \qquad
        \operatorname{Var}(W)=\frac{\pi}{3}-1.
\]

More generally, by homogeneity, for every \(m\ge0\), the unique \(L^2\)
solution of \eqref{eq:W-smoothing} with mean \(m\) is \(mW\).
Equation \eqref{eq:W-smoothing} belongs to the class of smoothing transforms,
or random weighted mean equations, studied by Liu~\cite{Liu01}. General
results from that work yield further information on the law of \(W\), which
we do not pursue here.
\end{remark}

\subsection{Paternal negligibility}

We now compare the initial advantaged genealogy with the RRT. For \(1\le k\le N-1\), set
\[
        T_k:=S_{k+1}-1.
\]
Since \(B\) is constant during the stable part of the block \([S_k,S_{k+1})\), its jump is
\begin{equation*}
        B_{k+1}-B_k
        =
        \frac12\left(
        A_{T_k}(\mu_{T_k},1)+A_{T_k}(\pi_{T_k},1)
        \right),
\end{equation*}
where \(\mu_{T_k}\) is the advantaged mother and \(\pi_{T_k}\) is the uniformly chosen father.

Define the maternal-only weights \(A^{(\mathrm m)}\) on the same pedigree by
\[
        A^{(\mathrm m)}_0(i,1):=\mathbf 1_{\{i=1\}},
\]
and
\begin{equation*}
        A^{(\mathrm m)}_{n+1}(i,1)
        =
        \begin{cases}
        A^{(\mathrm m)}_n(i,1), & i\neq \kappa_n,\\[3pt]
        \dfrac12 A^{(\mathrm m)}_n(\mu_n,1), & i=\kappa_n.
        \end{cases}
\end{equation*}
Let
\[
        B_k^{(\mathrm m)}
        :=
        \sum_{i\in\mathcal Y_{S_k}}A^{(\mathrm m)}_{S_k}(i,1).
\]

Thus \(A^{(\mathrm m)}\) is constructed on exactly the same population
pedigree as \(A\), but with all paternal genetic contributions discarded:
at each reproduction event, the newborn retains only the half of the
founder's contribution inherited from its mother. In particular, the
type dynamics and the set of advantaged individuals are unchanged.
Since advantaged individuals are never replaced and each
new advantaged individual is born from an advantaged mother, the maternal
genealogy of the advantaged class couples to a growing random recursive tree.

\begin{lemma}[Maternal-only coupling]\label{lem:maternal-equals-RRT}
For every \(1\le k\le N\),
\[
        0\le B_k^{(\mathrm m)}\le B_k.
\]
Moreover,
\[
        \bigl(B_k^{(\mathrm m)}\bigr)_{1\le k\le N}
        \stackrel d=
        \bigl(B_k^\star\bigr)_{1\le k\le N}.
\]
\end{lemma}

\begin{proof}
The inequality follows by induction from
\begin{equation*}
        0\le A^{(\mathrm m)}_n(i,1)\le A_n(i,1).
\end{equation*}

Now label the advantaged individuals by their order of appearance and attach each new individual to its advantaged mother. Fix \(1\le k\le N-1\) and write \(L_k:=S_{k+1}-S_k-1\). On \(\{L_k=\ell\}\), conditionally on \(\mathcal F_{S_k+\ell}\), the mother producing the jump \(k\to k+1\) is uniform on the \(k\) advantaged individuals. Hence, conditionally on the genealogy up to size \(k\), the parent of vertex \(k+1\) is uniform on \(\{1,\ldots,k\}\), exactly as in a random recursive tree.

Moreover, stable steps do not modify the maternal-only weights carried by advantaged individuals, while a child receives half the weight of its mother. Thus a vertex \(v\) carries weight
\(2^{-\operatorname{depth}(v)}\), and therefore
\begin{equation*}
        B^{(\mathrm m)}_k
        =
        \sum_{v\in\mathcal T_k}
        2^{-\operatorname{depth}(v)}.
\end{equation*}
Consequently,
\begin{equation*}
        \bigl(B^{(\mathrm m)}_k\bigr)_{1\le k\le N}
        \stackrel{d}{=}
        \bigl(B_k^*\bigr)_{1\le k\le N}.
\end{equation*}
\end{proof}

We shall use the following simple estimate on the disadvantaged mass immediately before a
jump.

\begin{lemma}\label{lem:C-before-jump}
There exists a constant \(K<\infty\), independent of \(N\) and \(k\), such that,
for all \(1\le k\le N-2\),
\[
        \mathbb E[C_{T_k}]
        \le
        K\sqrt{k+1}.
\]
\end{lemma}

\begin{proof}
Recall that $L:=S_{k+1}-S_k-1$.

For every \(\ell\ge0\), on \(\{L=\ell\}\) one has
\(T_k=S_k+\ell\). Conditionally on
\(\mathcal F_{S_k+\ell}\) and on \(\{L=\ell\}\), the final mother is
advantaged. From Proposition \ref{prop:fund-dynamics-BC} (Equation \eqref{eq-Cbar-jump-exp}),
\[
        \mathbb E\left[
        C_{S_{k+1}}
        \,\middle|\,
        \mathcal F_{S_k+\ell},\,L=\ell
        \right]
        =
        \left(1-\frac1{N-k}\right)C_{S_k+\ell}.
\]
Multiplying by \(\mathbf 1_{\{L=\ell\}}\), taking expectations, and
summing over \(\ell\ge0\), we obtain
\[
\begin{aligned}
        \mathbb E[C_{S_{k+1}}]
        &=
        \left(1-\frac1{N-k}\right)
        \sum_{\ell\ge0}
        \mathbb E\left[
        C_{S_k+\ell}\mathbf 1_{\{L=\ell\}}
        \right] \\
        &=
        \left(1-\frac1{N-k}\right)\mathbb E[C_{T_k}].
\end{aligned}
\]
Hence, since \(k\le N-2\),
\[
        \mathbb E[C_{T_k}]
        =
        \frac{N-k}{N-k-1}\mathbb E[C_{S_{k+1}}]
        \le
        2\,\mathbb E[C_{S_{k+1}}].
\]
The conclusion follows from Lemma~\ref{lem:uv-apriori}, after increasing
\(K\) if necessary.
\end{proof}

The following lemma shows that, up to any mesoscopic level \(k_0\ll N\), paternal genetic input is asymptotically negligible. Thus the genetic contribution carried by the advantaged class is, to first order, governed by the maternal-only dynamics.

\begin{lemma}[Negligibility of paternal input]\label{lem:paternal-error}
Let \(B^{(\mathrm m)}_k\) denote the maternal-only contribution. There exists
a constant \(K<\infty\) such that, for every sequence \(k_0=k_0(N)\) satisfying
\(k_0\to\infty\) and \(k_0/N\to0\),
\[
        \mathbb E\left[
        B_{k_0}-B^{(\mathrm m)}_{k_0}
        \right]
        \le
        K\frac{k_0^{3/2}}{N}.
\]
Consequently,
\[
        \frac{B_{k_0}-B^{(\mathrm m)}_{k_0}}{\sqrt{k_0}}
        \xrightarrow[N\to\infty]{\mathbb P}0 .
\]
\end{lemma}

\begin{proof}
Set
\[
        E_k:=B_k-B_k^{(\mathrm m)}.
\]
Stable steps do not modify the contribution carried by advantaged
individuals. Hence, at the jump from \(k\) to \(k+1\),
\[
\begin{aligned}
        E_{k+1}-E_k
        &=
        \frac12\left(
        A_{T_k}(\mu_{T_k},1)
        -
        A_{T_k}^{(\mathrm m)}(\mu_{T_k},1)
        \right)
        +
        \frac12 A_{T_k}(\pi_{T_k},1).
\end{aligned}
\]
Conditionally on the pre-jump configuration, the mother is uniform on the \(k\) advantaged individuals, whereas the father is uniform on the whole population. Therefore
\[
        \mathbb E\left[
        E_{k+1}
        \,\middle|\,
        \mathcal F_{S_k}
        \right]
        =
        \left(1+\frac{1}{2k}\right)E_k
        +
        \frac{1}{2N}
        \mathbb E\left[
        B_k+C_{T_k}
        \,\middle|\,
        \mathcal F_{S_k}
        \right].
\]
Now from Lemma~\ref{lem:uv-apriori} and Lemma~\ref{lem:C-before-jump}, there exists \(K<\infty\) such that, uniformly in \(1\le k<N-1\),
\begin{equation*}
        \mathbb E\left[
        B_k + C_{T_{k}}
        \right]
        \le K\sqrt{k}.
\end{equation*}
Taking expectations, we now deduce that
\begin{equation}\label{eq:D-recursion}
        \mathbb E[E_{k+1}]
        \le
        \left(1+\frac{1}{2k}\right)\mathbb E[E_k]
        +
        K\frac{\sqrt{k}}{N}.
\end{equation}
Recall that $\gamma_k=
        \prod_{\ell=1}^{k-1}
        \left(1+\frac1{2\ell}\right)$ from \eqref{eq:RRT-weight}.
Dividing \eqref{eq:D-recursion} by \(\gamma_{k+1}\) gives
\[
        \frac{\mathbb E[E_{k+1}]}{\gamma_{k+1}}
        \le
        \frac{\mathbb E[E_k]}{\gamma_k}
        +
        K\frac{\sqrt{k}}{N\gamma_{k+1}} .
\]
Since \(E_1=0\), iteration yields
\[
        \frac{\mathbb E[E_{k_0}]}{\gamma_{k_0}}
        \le
        \frac{K}{N}
        \sum_{k=1}^{k_0-1}
        \frac{\sqrt{k}}{\gamma_{k+1}} .
\]
By Lemma~\ref{lem:gamma-asymptotics}, there exists a finite number $K$ such that for any integer $k\geq 1$, $\sqrt{k}/\gamma_{k+1}\leq \sqrt{k+1}/\gamma_{k+1} \leq K$, therefore such that
\[
        \sum_{k=1}^{k_0-1}
        \frac{\sqrt{k}}{\gamma_{k+1}}
        \le
        K k_0,
\]
and thus
\[
        \mathbb E[E_{k_0}]
        \le
        K\frac{\gamma_{k_0}k_0}{N}
        \le
        K\frac{k_0^{3/2}}{N},
\] by Lemma~\ref{lem:gamma-asymptotics} again.
This proves the announced expectation bound.

Finally, \(E_{k_0}\ge0\), so Markov's inequality gives, for every \(\eta>0\),
\[
        \mathbb P\left(
        \frac{E_{k_0}}{\sqrt{k_0}}>\eta
        \right)
        \le
        \frac{\mathbb E[E_{k_0}]}{\eta\sqrt{k_0}}
        \le
        \frac{K}{\eta}\frac{k_0}{N},
\]
which tends to \(0\) as $N$ goes to infinity, by assumption. Hence
\[
        \frac{B_{k_0}-B^{(\mathrm m)}_{k_0}}{\sqrt{k_0}}
        \xrightarrow[N\to\infty]{\mathbb P}0 .
\]
\end{proof}

\subsection{Seeding}

\begin{proposition}[Seeding at mesoscopic entrance]\label{prop:B-seeding}
Let \(k_0=k_0(N)\) be such that $1\ll k_0\ll N.$ Then
\begin{equation}\label{eq-B_{k_0}}
        \frac{B_{k_0}}{\sqrt{k_0}}
        \Longrightarrow
        \frac{2}{\sqrt{\pi}}\,M_\infty^\star,
        \qquad \text{ as } N\to\infty.
\end{equation}
Moreover,
\begin{equation}\label{eq-M_{k_0}}
        M_{k_0}^{(N)}
        \Longrightarrow
        M_\infty^\star.
\end{equation}
\end{proposition}

\begin{proof}
By Lemma~\ref{lem:maternal-equals-RRT},
\[
        \frac{B_{k_0}^{(\mathrm m)}}{\sqrt{k_0}}
        \stackrel{d}{=}
        \frac{B_{k_0}^{\star}}{\sqrt{k_0}}
        =
        \frac{\gamma_{k_0}}{\sqrt{k_0}}\,M_{k_0}^{\star}.
\]
Lemma~\ref{lem:gamma-asymptotics} gives
\[
        \frac{\gamma_{k_0}}{\sqrt{k_0}}
        \longrightarrow
        \frac{2}{\sqrt{\pi}},
\]
and from Proposition \ref{prop:RRT-mart}, \(M_k^\star\to M_\infty^\star\) in \(L^2\). Hence
\[
        \frac{B_{k_0}^{(\mathrm m)}}{\sqrt{k_0}}
        \Longrightarrow
        \frac{2}{\sqrt{\pi}}\, M_\infty^\star .
\]
Together with Lemma~\ref{lem:paternal-error}, this implies \eqref{eq-B_{k_0}}. It remains to identify the entrance value of the martingale. We write
\[
        M_{k_0}^{(N)}
        =
        a_{k_0,N}
        \left(
        \frac{B_{k_0}}{k_0}
        -
        \frac{C_{k_0}}{N-k_0}
        \right).
\]
By Lemma \ref{lem:a-estimates}, since \(k_0=k_0(N)\) diverges as $N$ goes to infinity,
\[
        \frac{a_{k_0,N}}{\sqrt{k_0}}
        \longrightarrow
        \frac{\sqrt{\pi}}{2}.
\]
Therefore,
\[
        a_{k_0,N}\frac{B_{k_0}}{k_0}
        =
        \frac{a_{k_0,N}}{\sqrt{k_0}}\,
        \frac{B_{k_0}}{\sqrt{k_0}}
        \Longrightarrow
        M_\infty^\star .
\]
On the other hand, Lemma~\ref{lem:uv-apriori} gives
\[
        \mathbb E\left[
        a_{k_0,N}\frac{C_{k_0}}{N-k_0}
        \right]
        \le
        K\frac{k_0}{N}
        \longrightarrow 0.
\]
Since $a_{k_0,N}\frac{C_{k_0}}{N-k_0}$ is non-negative, it converges to \(0\) in probability, which gives \eqref{eq-M_{k_0}}.
\end{proof}

\section{Macroscopic propagation}\label{sec:macroscopic}

\subsection{Quadratic variation}

\begin{proposition}[One-step conditional bound]
\label{prop:QV-local-corr}
There exists a constant \(K<\infty\), independent of \(N\) and \(k\),
such that, for all \(N\ge3\) and \(1\le k\le N-2\), setting
\(d:=N-k\), one has
\begin{equation*}
\begin{aligned}
        \mathbb E\left[
        \bigl(M_{k+1}^{(N)}-M_k^{(N)}\bigr)^2
        \,\middle|\,
        \mathcal F_{S_k}
        \right]
        \le
        K\left[
        \frac{\overline B_k+\overline C_k}{k}
        +
        \frac{\overline B_k+\overline C_k}{d}
        +
        \frac{k}{d^2}
        \bigl(\overline B_k+\overline C_k\bigr)
        \right].
\end{aligned}
\end{equation*}
\end{proposition}

We use the following Lemma to prove Proposition \ref{prop:QV-local-corr}.

\begin{lemma}[Disadvantaged average during a stable block]
\label{lem:stable-block-average}
Fix \(1\le k\le N-1\), and set
\[
        m_j
        :=
        \mathbb E\left[
        \overline C^{(j)}
        \,\middle|\,
        \mathcal F_{S_k},
        L\ge j
        \right],
        \qquad j\ge0.
\]
Then
\begin{equation}\label{eq:rec-m_j}
        m_{j+1}
        =
        (1-\lambda)m_j+\lambda\overline B_k,
        \qquad
        m_0=\overline C_k.
\end{equation}
Consequently,
\begin{equation}\label{eq:sol-m_j}
        m_j
        =
        \overline B_k
        +
        (1-\lambda)^j
        \bigl(
        \overline C_k-\overline B_k
        \bigr),
\end{equation}
and in particular
\begin{equation}\label{eq-control-m}
        m_j
        \le
        \max\{\overline B_k,\overline C_k\}
        \le
        \overline B_k+\overline C_k.
\end{equation}
\end{lemma}

\begin{proof}[Proof of Lemma \ref{lem:stable-block-average}]
On \(\{L\geq j+1\}\), Proposition~\ref{prop:fund-dynamics-BC} (Equation \eqref{eq-Cbar-stable-exp}) gives
\[
        \mathbb E\left[
        \overline C^{(j+1)}
        \,\middle|\,
        \mathcal F_{S_k+j},\,L\geq j+1
        \right]
        =
        \overline C^{(j)}+\lambda X^{(j)}
        =
        (1-\lambda)\overline C^{(j)}
        +
        \lambda\overline B_k.
\]
Moreover, on \(\{L\ge j\}\), the conditional probability that the
next update is stable is \(q\), independently of the current genetic
configuration. Hence
\[
\mathbb E\left[
        \overline C^{(j)}
        \,\middle|\,
        \mathcal F_{S_k},
        L\ge j+1
        \right]
=
\mathbb E\left[
        \overline C^{(j)}
        \,\middle|\,
        \mathcal F_{S_k},
        L\ge j
        \right]=m_j.\]
Taking conditional expectations in the first display yields \eqref{eq:rec-m_j}. The explicit formula \eqref{eq:sol-m_j} follows by solving this affine recurrence. The first inequality of \eqref{eq-control-m} can be obtained by a disjunction case based on the sign of $\overline{C}_k-\overline{B}_k$ and the second inequality comes from the positivity of the mean contributions $\overline{C}_k$ and $\overline{B}_k$.
\end{proof}

\begin{proof}[Proof of Proposition \ref{prop:QV-local-corr}]
Throughout the proof, \(K\) denotes a finite constant whose value may
change from line to line. 

Since by definition \eqref{def:akN}, \(a_{k,N}=a_{k+1,N}\alpha_{k,N}\) for every $1\leq k\leq N-1$, then from the definition \eqref{def-M} of $M_k^{(N)}$,
\begin{equation}\label{eq-def-diff-M}
        M_{k+1}^{(N)}-M_k^{(N)}
        =
        a_{k+1,N}
        \bigl(X_{k+1}-\alpha_{k,N}X_k\bigr).
\end{equation}

By Lemma~\ref{lem:a-estimates}, there exists $K\geq0$ such that for any $1\leq k\leq N-2$,
\begin{equation}
\label{eq:QV-a-bound}
        a_{k+1,N}^2\le Kk.
\end{equation}
We now study the difference $X_{k+1}-\alpha_{k,N}X_k$, and to this aim analyse the block \([S_k,S_{k+1})\). Recall the block notations \eqref{eq:notations} and set $T:=S_k$ for a better readability.
Conditionally on \(\mathcal F_T\),
\begin{equation*}
        \mathbb P(L=\ell\mid\mathcal F_T)
        =q^\ell\frac{k}{N},
        \qquad \ell\ge0,
\end{equation*}
and therefore
\begin{equation}
\label{eq:QV-geometric-moments}
        \mathbb E[L\mid\mathcal F_T]=\frac dk,
        \qquad
        \operatorname{Var}(L\mid\mathcal F_T)
        =\frac{Nd}{k^2}.
\end{equation}
Set also
\begin{equation*}
        \lambda:=\frac{k}{2Nd},
        \qquad
        r:=1-\lambda,
        \qquad
        s:=1-\frac{d}{2N(k+1)}.
\end{equation*}
For \(j\ge0\), let
\begin{equation*}
        \mathcal B_j
        :=
        \mathcal F_T
        \vee
        \sigma\bigl(
        (\mu_{T+i},\pi_{T+i},\kappa_{T+i})
        :0\le i<j
        \bigr).
\end{equation*}
On \(\{L\ge j\}\), denote by \(X^{(j)}\) and \(\overline C^{(j)}\) the values of \(X\) and \(\overline C\) after the first \(j\) stable steps of the block i.e. at time $T+j$. 

On \(\{L\ge j+1\}\), Proposition \ref{prop:fund-dynamics-BC} gives
\begin{equation*}
        X^{(j+1)}
        =
        rX^{(j)}+\xi_{j+1},
        \qquad \text{where}\qquad
        \mathbb E\left[
        \xi_{j+1}
        \,\middle|\,
        \mathcal B_j,L\ge j+1
        \right]
        =0.
\end{equation*}
At the final jump, on \(\{L=\ell\}\), Proposition \ref{prop:fund-dynamics-BC} again gives
\begin{equation*}
        X_{k+1}
        =
        sX^{(\ell)}+\zeta,
        \qquad \text{where\qquad}
        \mathbb E\left[
        \zeta
        \,\middle|\,
        \mathcal B_\ell,L=\ell
        \right]
        =0.
\end{equation*}
It follows that
\begin{equation}\label{decomp-diff-X}
\begin{aligned}
        X_{k+1}-\alpha_{k,N}X_k
        &=
        (sr^L-\alpha_{k,N})X_k+s\sum_{j=1}^{L}r^{L-j}\xi_j+\zeta\\&
        =P+Q+J,
        \end{aligned}
\end{equation}
where
\begin{equation*}
        P:=(sr^L-\alpha_{k,N})X_k,
        \qquad
        Q:=s\sum_{j=1}^{L}r^{L-j}\xi_j,
        \qquad
        J:=\zeta.
\end{equation*}
Since
\begin{equation}\label{eq:PQJ}
        (P+Q+J)^2
        \le
        3P^2+3Q^2+3J^2,
\end{equation} the three terms can be dealt with separately.

We first control \(P\). Note that,
\begin{equation*}
        \mathbb E[sr^L\mid\mathcal F_T]
        =
        \frac{s(k/N)}{1-qr}
        =
        \frac{2Nk+N+k}{(2N+1)(k+1)}
        =
        \alpha_{k,N}.
\end{equation*}
Using \eqref{eq:QV-geometric-moments}, since \(0<s\le1\) and
\(\ell\mapsto r^\ell\) is \(\lambda\)-Lipschitz,
\begin{equation}
\label{eq:QV-random-coefficient-variance}
        \operatorname{Var}(sr^L\mid\mathcal F_T)
        \le
        \lambda^2\operatorname{Var}(L\mid\mathcal F_T)
        \le
        \frac{K}{Nd}.
\end{equation}
Furthermore, since \(0\le\overline B_k,\overline C_k\le1\),
\begin{equation}\label{eq:QV-X-square-bound}
        X_k^2
        =
        (\overline B_k-\overline C_k)^2
        \le
        \overline B_k+\overline C_k.
\end{equation}
Hence, by \eqref{eq:QV-a-bound},
\eqref{eq:QV-random-coefficient-variance}, and
\eqref{eq:QV-X-square-bound},
\begin{equation}
\label{eq:QV-P-bound}
        a_{k+1,N}^2
        \mathbb E[P^2\mid\mathcal F_T]
        \le
        K\frac{k}{Nd}
        \bigl(\overline B_k+\overline C_k\bigr)\leq K\frac{\overline B_k+\overline C_k}{d}
\end{equation}
since $k\leq N$. We next control the stable steps martingale term $Q$. During a stable step, recall from the fundamental dynamics described in Proposition \ref{prop:fund-dynamics-BC} that
\begin{equation*}
        X^{(j+1)}-X^{(j)}
        =
        -\frac1d
        \left(
        \frac{w(\mu)+w(\pi)}{2}-w(\kappa)
        \right).
\end{equation*}
Since by definition
\begin{equation}\label{eq-def-xi}
        \xi_{j+1}
        =
        X^{(j+1)}-X^{(j)}+\lambda X^{(j)}
\end{equation}
and $\mathbb{E}\left[
        \xi_{j+1}
        \,\middle|\,
        \mathcal B_j,L\ge j+1
        \right]
        =0$,
\begin{align}
\label{eq-detail-momentXi}
        \mathbb E\left[
        \xi_{j+1}^2
        \,\middle|\,
        \mathcal B_j,L\ge j+1
        \right]
        &=\operatorname{Var}\left(
        \xi_{j+1}
        \,\middle|\,
        \mathcal B_j,L\ge j+1
        \right) \notag\\
        &=\operatorname{Var}\left(
        X^{(j+1)}-X^{(j)}
        \,\middle|\,
        \mathcal B_j,L\ge j+1
        \right) \notag\\
        &\le
        \frac{1}{d^2}
        \mathbb{E}\left[]
        \Big(
        \frac{w(\mu)+w(\pi)}{2}-w(\kappa)
        \Big)^2
        \middle|\,
        \mathcal B_j,L\ge j+1
        \right]\notag\\
        &\le
        \frac{K}{d^2}
        \mathbb{E}\left[]
        \frac{w(\mu)+w(\pi)}{2}-w(\kappa)
        \middle|\,
        \mathcal B_j,L\ge j+1
        \right]\qquad\text{since $w(i)^2\leq w(i)$} \notag\\
        &\le
        \frac{K}{d^2}
        \left(
        \overline C^{(j)}+\overline B_k
        \right).
\end{align}
Indeed, conditionally on a stable step, the mother and the killed
individual are uniform on the disadvantaged class, whereas the father
is uniform on \(I_N\), and \(w(i)^2\le w(i)\).

Conditionally on \(L=\ell\), the increments
\(\xi_1,\ldots,\xi_\ell\) are martingale differences, so the cross
terms vanish. Since \(0\le r,s\le1\), \eqref{eq-detail-momentXi}
gives
\begin{equation*}
        \mathbb E[Q^2\mid\mathcal F_T]
        \le
        \frac{K}{d^2}
        \mathbb E\left[
        \sum_{j=0}^{L-1}
        \left(
        \overline C^{(j)}+\overline B_k
        \right)
        \,\middle|\,
        \mathcal F_T
        \right].
\end{equation*}
Let
\begin{equation*}
        m_j
        :=
        \mathbb E\left[
        \overline C^{(j)}
        \,\middle|\,
        \mathcal F_T,L\ge j
        \right].
\end{equation*}
By Lemma~\ref{lem:stable-block-average},
\begin{equation*}
        m_j\le \overline B_k+\overline C_k.
\end{equation*}
Since $\mathbb E\left[
        \overline C^{(j)}
        \,\middle|\,
        \mathcal F_T,L\ge j+1
        \right]=\mathbb E\left[
        \overline C^{(j)}
        \,\middle|\,
        \mathcal F_T,L\ge j
        \right]$,
\begin{align*}
        \mathbb E\left[
        \sum_{j=0}^{L-1}
        \left(
        \overline C^{(j)}+\overline B_k
        \right)
        \,\middle|\,
        \mathcal F_T
        \right]&= \sum_{j\geq0}
        E\left[\mathbf{1}_{L\geq j+1}
        \left(
        \overline C^{(j)}+\overline B_k
        \right)
        \,\middle|\,
        \mathcal F_T
        \right]
        \\&=\sum_{j\geq0}q^{j+1}
        E\left[
        \left(
        \overline C^{(j)}+\overline B_k
        \right)
        \,\middle|\,
        \mathcal F_T, L\geq j+1
        \right]\\&=
        \sum_{j\ge0}q^{j+1}
        \left(m_j+\overline B_k\right)\\
        &\le
        K\frac dk
        \left(\overline B_k+\overline C_k\right).
\end{align*}
Therefore using that $q=d/N$,
\begin{equation*}
        \mathbb E[Q^2\mid\mathcal F_T]
        \le
        \frac{K}{kd}
        \left(\overline B_k+\overline C_k\right)
\end{equation*}
and hence, by \eqref{eq:QV-a-bound},
\begin{equation}
\label{eq:QV-Q-bound}
        a_{k+1,N}^2
        \mathbb E[Q^2\mid\mathcal F_T]
        \le
        \frac{K}{d}
        \left(\overline B_k+\overline C_k\right).
\end{equation}

It remains to control the final jump term $J=\zeta=X_{k+1}-sX^{(\ell)}$. Let \(\overline C_-\) denote the disadvantaged average immediately before that jump. Proposition \ref{prop:fund-dynamics-BC} gives
\begin{equation*}
        X_{k+1}-X^{(L)}
        =
        -\frac{\overline B_k}{k+1}
        -\frac{\overline C_-}{d-1}
        +
        \frac{w(\mu)+w(\pi)}{2(k+1)}
        +
        \frac{w(\kappa)}{d-1}.
\end{equation*}
We have as previously
\begin{equation*}
\begin{aligned}
        \mathbb E\left[
        J^2
        \,\middle|\,
        \mathcal B_\ell,L=\ell
        \right]
        &=Var(J \,\big|\,
        \mathcal B_\ell,L=\ell)\\&=Var(X_{S_{k+1}}-X_{S_{k+1}-1} \,\big|\,
        \mathcal B_\ell,L=\ell)\\&=Var\Big( \frac{w(\mu)+w(\pi)}{2(k+1)}
        +\frac{w(\kappa)}{d-1} \,\big|\,
        \mathcal B_\ell,L=\ell\Big) \\&\le
        K\left[
        \frac{\overline B_k+\overline C_-}{k^2}
        +
        \frac{\overline C_-}{d^2}
        \right] \qquad\text{using again \(w(i)^2\le w(i)\) and $k\leq N-2$.}
    \end{aligned}
\end{equation*}
Since \(\overline C_-=\overline C^{(L)}\), conditioning on \(L\) and using
Lemma~\ref{lem:stable-block-average} gives
\begin{equation*}
        \mathbb E[\overline C_-\mid\mathcal F_T]
        =
        \sum_{\ell\ge0}q^\ell\frac{k}{N}\,m_\ell
        \le
        \overline B_k+\overline C_k.
\end{equation*}
Consequently,
\begin{equation*}
        \mathbb E[J^2\mid\mathcal F_T]
        \le
        K\left[
        \frac{\overline B_k+\overline C_k}{k^2}
        +
        \frac{\overline B_k+\overline C_k}{d^2}
        \right],
\end{equation*}
and therefore
\begin{equation}
\label{eq:QV-J-bound}
        a_{k+1,N}^2
        \mathbb E[J^2\mid\mathcal F_T]
        \le
        K\left[
        \frac{\overline B_k+\overline C_k}{k}
        +
        \frac{k}{d^2}
        \bigl(\overline B_k+\overline C_k\bigr)
        \right].
\end{equation}

Recalling the Equations \eqref{eq-def-diff-M}, \eqref{decomp-diff-X}, \eqref{eq:PQJ} and using \eqref{eq:QV-P-bound},
\eqref{eq:QV-Q-bound}, and \eqref{eq:QV-J-bound}, proves the result.
\end{proof}

\begin{lemma}[Negligibility of the martingale increments away from the origin]
\label{lem:M-flat-away-origin}
Let \(k_0=k_0(N)\) be an integer sequence such that $1\ll k_0\ll N$.
Then, for every \(\varepsilon\in(0,1)\),
\[
        \sup_{k_0\le k\le \lfloor(1-\varepsilon)N\rfloor}
        \left|
        M_k^{(N)}-M_{k_0}^{(N)}
        \right|
        \xrightarrow[N\to\infty]{L^2}0.
\]
In particular, the convergence holds in probability.
\end{lemma}

\begin{proof}
Write
\[
        \Delta M_k:=M_{k+1}^{(N)}-M_k^{(N)}.
\]
Set
\[
        m_N:=\lfloor(1-\varepsilon)N\rfloor.
\]
For \(k\le m_N\), we have
\[
        d=N-k
        \ge N-(1-\varepsilon)N
        =\varepsilon N
        \ge \varepsilon k.
\]
Consequently,
\[
        \frac{1}{d}\le \frac{1}{\varepsilon k},
        \qquad
        \frac{k}{d^2}\le \frac{1}{\varepsilon^2 k}.
\]
Taking expectations in Proposition~\ref{prop:QV-local-corr} and then
using Lemma~\ref{lem:uv-apriori}, we obtain, for \(k\le m_N-1\),
\[
\begin{aligned}
        \mathbb E\left[(\Delta M_k)^2\right]
        &\le
        K_\varepsilon\,
        \mathbb E\left[
        \frac{\overline B_k+\overline C_k}{k}
        \right] \\
        &\le
        \frac{6K_\varepsilon}{k^{3/2}},
\end{aligned}
\]
where \(K_\varepsilon<\infty\) is independent of \(N\) and \(k\).
Therefore,
\[
\begin{aligned}
        \sum_{k=k_0}^{m_N-1}
        \mathbb E\left[(\Delta M_k)^2\right]
        &\le
        K_\varepsilon
        \sum_{k=k_0}^{\infty}\frac{1}{k^{3/2}} \\
        &\le
        \frac{K_\varepsilon}{\sqrt{k_0}}.
\end{aligned}
\]
The right-hand side tends to \(0\) as $N$ goes to infinity.

Since \((M_k^{(N)})_k\) is a martingale with respect to
\((\mathcal F_{S_k})_k\), Doob's \(L^2\) inequality gives
\[
        \mathbb E\left[
        \sup_{k_0\le k\le m_N}
        \left|
        M_k^{(N)}-M_{k_0}^{(N)}
        \right|^2
        \right]
        \le
        4\,
        \mathbb E\left[
        \left|
        M_{m_N}^{(N)}-M_{k_0}^{(N)}
        \right|^2
        \right]=
        4
        \sum_{k=k_0}^{m_N-1}
        \mathbb E\left[(\Delta M_k)^2\right],
\]
which tends to \(0\). This proves the claim.
\end{proof}

\subsection{\texorpdfstring{$B$}{B} and \texorpdfstring{$C$}{C} reconstruction}

The exact decomposition \eqref{eq:Cbar-drift-decomposition} identifies
the drift of \(\overline C\) over the whole block \([S_k,S_{k+1}]\).
We now estimate the conditional variance of its centred remainder.

\begin{proposition}[Quadratic bound for the \(\overline C\)-remainder]
\label{prop:Cbar-remainder-variance}
There exists a constant \(K<\infty\) such that, for every \(N\ge3\)
and \(1\le k\le N-2\), setting \(d:=N-k\), one has
\begin{equation}
\label{eq:Cbar-remainder-global-bound}
        \mathbb E\left[
        (\Delta R_{k+1})^2
        \,\middle|\,
        \mathcal F_{S_k}
        \right]
        \le
        K\left(
        \frac{1}{kd}
        +
        \frac{1}{d^2}
        \right)
        \left(\overline B_k+\overline C_k\right),
\end{equation}
where \(\Delta R_{k+1}\) is the centred increment defined in
\eqref{eq:Cbar-drift-decomposition}.

Consequently, for every \(\varepsilon\in(0,1)\), for $N>2/\varepsilon$, there exists
\(K_\varepsilon<\infty\) such that, whenever
\begin{equation*}
        1\le k\le\lfloor(1-\varepsilon)N\rfloor,
\end{equation*}
one has
\begin{equation}
\label{eq:Cbar-remainder-away-fixation}
        \mathbb E\left[
        (\Delta R_{k+1})^2
        \,\middle|\,
        \mathcal F_{S_k}
        \right]
        \le
        \frac{K_\varepsilon}{Nk}
        \left(\overline B_k+\overline C_k\right).
\end{equation}
\end{proposition}

\begin{proof}
The proof uses the same type of arguments and the same outline than that of Proposition \ref{prop:QV-local-corr}. We use the block notation \eqref{eq:notations} as well as, as in the proof of
Proposition~\ref{prop:QV-local-corr}, the notations $T:=S_k$ and
$\mathcal B_j:=\mathcal F_T\vee\sigma\bigl(
(\mu_{T+i},\pi_{T+i},\kappa_{T+i}):0\le i<j\bigr)$ for all $j\ge0$.
Similarly, on \(\{L\ge j\}\), \(X^{(j)}\) and
\(\overline C^{(j)}\) denote the values of \(X\) and
\(\overline C\) after the first \(j\) stable steps of the block, i.e. at time $S_k+j=T+j$.

Recall that \begin{equation}\label{eq-R-rappel}
        \Delta R_{k+1}=\overline C_{k+1}-\overline C_k-\frac{1}{2N+1}X_k.\end{equation}

Conditionally on \(\mathcal F_T\), recall from Proposition \ref{prop:fund-dynamics-Y} that $L$ follows a geometric law with parameter $1-q$, with values in $\mathbb{Z}_+$. In particular,
\begin{equation}
\label{eq:Cbar-rem-geometric-moments}
        \mathbb E[L^2\mid\mathcal F_T]
        \le
        \frac{KNd}{k^2}.
\end{equation}

During a stable step, write as in Equation \eqref{eq-def-xi}
\begin{equation}\label{eq-diff-Cj}
        \overline C^{(j+1)}-\overline C^{(j)}=-(X^{(j+1)}-X^{(j)})
        =
        \lambda X^{(j)}-\xi_{j+1}.
\end{equation}
Then, on \(\{L\ge j+1\}\),
$\mathbb E\left[\xi_{j+1}\,\middle|\,\mathcal B_j,L\ge j+1
\right]=0,$
and from Equation \eqref{eq-detail-momentXi}
\begin{equation}
\label{eq:Cbar-rem-stable-local-variance}
        \mathbb E\left[
        \xi_{j+1}^2
        \,\middle|\,
        \mathcal B_j,L\ge j+1
        \right]
        \le
        \frac{K}{d^2}
        \left(
        \overline C^{(j)}+\overline B_k
        \right).
\end{equation}

At the final jump, on \(\{L=\ell\}\), set
\[
        J
        :=
        \overline C_{k+1}-\overline C^{(\ell)}
        =
        \frac{\overline C^{(\ell)}-w(\kappa)}{d-1}
 \qquad\text{from Proposition \ref{prop:fund-dynamics-BC}.}\]

Set
\[
        P
        :=
        \sum_{j=0}^{L-1}\lambda X^{(j)},
        \qquad
        Q
        :=
        \sum_{j=0}^{L-1}\xi_{j+1}.
\]
Then by \eqref{eq-diff-Cj},
\[
        \overline C_{k+1}-\overline C_k=P-Q+J.
\]

From Proposition \ref{prop:fund-dynamics-BC} (Equations \eqref{eq-Cbar-stable-exp} and \eqref{eq-Cbar-jump-exp} respectively), the terms \(Q\) and \(J\) are conditionally centred. 
Hence, by Proposition~\ref{prop:M},
\[
        \mathbb E[P\mid\mathcal F_T]
        =
        \frac{1}{2N+1}X_k,
\]
and therefore by \eqref{eq-R-rappel},
\begin{equation}
\label{eq:Cbar-rem-decomposition}
        \Delta R_{k+1}
        =
        \bigl(P-\mathbb E[P\mid\mathcal F_T]\bigr)
        -Q+J.
\end{equation}
We again use the fact that $(P+Q+J)^2\leq 3P^2+3Q^2+3J^2$ and control each term separately.
We first control \(Q\). Let
\[
        m_j
        :=
        \mathbb E\left[
        \overline C^{(j)}
        \,\middle|\,
        \mathcal F_T,L\ge j
        \right].
\]
By Lemma~\ref{lem:stable-block-average},
\begin{equation}
\label{eq:Cbar-rem-mj-bound}
        m_j
        \le
        \overline B_k+\overline C_k,
        \qquad j\ge0.
\end{equation}

Writing
\[
        Q
        =
        \sum_{j\ge0}
        \mathbf 1_{\{L\ge j+1\}}\xi_{j+1},
\]
and using the martingale-difference property, \eqref{eq:Cbar-rem-stable-local-variance}, and \eqref{eq:Cbar-rem-mj-bound}, we obtain
\begin{align}
\label{eq:Cbar-rem-stable-bound}
        \mathbb E[Q^2\mid\mathcal F_T]
        &\le
        \frac{K}{d^2}
        \sum_{j\ge0}q^{j+1}
        \left(m_{j}+\overline B_k\right) \notag\\
        &\le
        \frac{K}{d^2}
        \left(\overline B_k+\overline C_k\right)
        \sum_{j\ge0}q^{j+1} \notag\\
        &=
        \frac{K}{kd}
        \left(\overline B_k+\overline C_k\right).
\end{align}

We now focus on the final jump term $J$. Conditionally on \(\mathcal B_\ell\) and \(\{L=\ell\}\),
$$\mathbb E[J\mid\mathcal B_\ell,L=\ell]=0,$$
and, since \(d\ge2\),
\begin{equation}
\label{eq:Cbar-rem-jump-local-variance}
        \mathbb E[J^2\mid\mathcal B_\ell,L=\ell]
        \le
        \frac{K}{d^2}\overline C^{(\ell)}.
\end{equation}
Therefore conditioning on \(L\) and using
\eqref{eq:Cbar-rem-mj-bound} gives
\begin{equation}
\label{eq:Cbar-rem-CL-bound}
        \mathbb E\left[
        \overline C^{(L)}
        \,\middle|\,
        \mathcal F_T
        \right]
        =
        \sum_{\ell\ge0}(1-q)q^\ell m_\ell
        \le
        \overline B_k+\overline C_k.
\end{equation}
Thus, \eqref{eq:Cbar-rem-jump-local-variance} and
\eqref{eq:Cbar-rem-CL-bound} give
\begin{equation}
\label{eq:Cbar-rem-jump-bound}
        \mathbb E[J^2\mid\mathcal F_T]
        \le
        \frac{K}{d^2}
        \left(\overline B_k+\overline C_k\right).
\end{equation}

It remains to control the fluctuation of \(P\). Since
\begin{equation}
\label{eq:Cbar-rem-P-centred-by-P2}
        \mathbb E\left[
        \left(
        P-\mathbb E[P\mid\mathcal F_T]
        \right)^2
        \,\middle|\,
        \mathcal F_T
        \right]
        \le
        \mathbb E[P^2\mid\mathcal F_T]
\end{equation}
and
\begin{equation*}
        (X^{(j)})^2
        \le
        \overline B_k+\overline C^{(j)},
\end{equation*} as in Equation \eqref{eq:QV-X-square-bound}.
Cauchy's inequality gives
\begin{equation}
\label{eq:Cbar-rem-P-cauchy}
        P^2
        \le
        \lambda^2L
        \sum_{j=0}^{L-1}
        \left(
        \overline B_k+\overline C^{(j)}
        \right).
\end{equation}
Conditioning on \(L\), using that this conditioning does not bias the
first stable steps, and applying
\eqref{eq:Cbar-rem-mj-bound}, \eqref{eq:Cbar-rem-geometric-moments},
and \eqref{eq:Cbar-rem-P-cauchy}, we get
\begin{align}
\label{eq:Cbar-rem-P2-bound}
        \mathbb E[P^2\mid\mathcal F_T]
        &\le
        K\lambda^2
        \left(\overline B_k+\overline C_k\right)
        \mathbb E[L^2\mid\mathcal F_T] \notag\\
        &\le
        \frac{K}{Nd}
        \left(\overline B_k+\overline C_k\right).
\end{align}
Since \(k\le N\), \eqref{eq:Cbar-rem-P-centred-by-P2} and
\eqref{eq:Cbar-rem-P2-bound} imply
\begin{equation}
\label{eq:Cbar-rem-drift-bound}
        \mathbb E\left[
        \left(
        P-\mathbb E[P\mid\mathcal F_T]
        \right)^2
        \,\middle|\,
        \mathcal F_T
        \right]
        \le
        \frac{K}{kd}
        \left(\overline B_k+\overline C_k\right).
\end{equation}

Combining \eqref{eq:Cbar-rem-decomposition},
\eqref{eq:Cbar-rem-stable-bound},
\eqref{eq:Cbar-rem-jump-bound}, and
\eqref{eq:Cbar-rem-drift-bound} proves \eqref{eq:Cbar-remainder-global-bound}.

Now fix \(\varepsilon\in(0,1)\) and assume
\(k\le\lfloor(1-\varepsilon)N\rfloor\). Then
\[
        d=N-k\ge\varepsilon N.
\]
Consequently,
\[
        \frac{1}{kd}
        \le
        \frac{1}{\varepsilon Nk},
        \qquad
        \frac{1}{d^2}
        \le
        \frac{1}{\varepsilon^2N^2}
        \le
        \frac{1}{\varepsilon^2Nk},
\]
because \(k\le N\). Applying these bounds in
\eqref{eq:Cbar-remainder-global-bound} proves
\eqref{eq:Cbar-remainder-away-fixation}.
\end{proof}

\begin{proposition}[Propagation through the deterministic phase]
\label{prop:deterministic-propagation}
Let \(k_0=k_0(N)\) be an integer sequence such that $
        1\ll k_0\ll N.$ Then, for every \(\varepsilon\in(0,1)\),
\begin{equation}\label{eq-equiv-C}
        \max_{k_0\le k\le \lfloor(1-\varepsilon)N\rfloor}
        \left|
        \frac{C_k}{\sqrt N}
        -
        \frac{2}{\sqrt\pi}
        \sqrt{\frac{k}{N}}
        \left(1-\frac{k}{N}\right)
        M_{k_0}^{(N)}
        \right|
        \xrightarrow[N\to\infty]{\mathbb P}0,
\end{equation}
and
\begin{equation}\label{eq-equiv-B}
        \max_{k_0\le k\le \lfloor(1-\varepsilon)N\rfloor}
        \left|
        \frac{B_k}{\sqrt N}
        -
        \frac{2}{\sqrt\pi}
        \sqrt{\frac{k}{N}}
        \left(1+\frac{k}{N}\right)
        M_{k_0}^{(N)}
        \right|
        \xrightarrow[N\to\infty]{\mathbb P}0.
\end{equation}
\end{proposition}

\begin{proof}
Fix \(\varepsilon\in(0,1)\), and write
\[
        m_N:=\lfloor(1-\varepsilon)N\rfloor,
        \qquad
        y_k:=\frac{k}{N}.
\]

Iterating \eqref{eq:Cbar-drift-decomposition}, for \(k\geq k_0\),
\begin{equation}\label{eq-decomp-C}
        \overline C_k
        =
        \overline C_{k_0}
        +
        \frac1{2N+1}
        \sum_{j=k_0}^{k-1}a_{j,N}^{-1}M_j^{(N)}
        +
        R_{k_0,k}.
        \qquad
        R_{k_0,k}:=\sum_{j=k_0}^{k-1}\Delta R_{j+1}.
\end{equation}
The process \((R_{k_0,k})_{k\ge k_0}\) is a martingale with respect to
\((\mathcal F_{S_k})_{k\ge k_0}\), starting from \(0\). By Doob's $L^2$ inequality,
\[
\begin{aligned}
        \mathbb E\left[
        \sup_{k_0\le k\le m_N}|R_{k_0,k}|^2
        \right]
        &\le
        4\sum_{j=k_0}^{m_N-1}
        \mathbb E[(\Delta R_{j+1})^2]                         \\
        &\le
        \frac{K_\varepsilon}{N}
        \sum_{j=k_0}^{m_N-1}
        \frac{\mathbb E[\overline B_j+\overline C_j]}{j}     \qquad\text{from Proposition~\ref{prop:Cbar-remainder-variance} and since $m_N=\lfloor(1-\epsilon)N\rfloor$}   \\
        &\le
        \frac{K_\varepsilon}{N}
        \sum_{j=k_0}^{\infty}\frac1{j^{3/2}}  \qquad\text{from Lemma~\ref{lem:uv-apriori}}\\
        &\le
        \frac{K_\varepsilon}{N\sqrt{k_0}} .
\end{aligned}
\]
Since \((N-k)/\sqrt N\le \sqrt N\) and $k_0\gg1$, it follows that
\begin{equation}\label{eq-control-R}
        \sup_{k_0\le k\le m_N}
        \frac{N-k}{\sqrt N}|R_{k_0,k}|
        \xrightarrow[N\to\infty]{\mathbb P}0.
\end{equation}

We now estimate the adapted part of the process $(\overline{C}_k)_{k_0\leq k\leq m_N}$. Write
\begin{equation}\label{eq-decom-aM}
\begin{aligned}
        \frac1{2N+1}
        \sum_{j=k_0}^{k-1}a_{j,N}^{-1}M_j^{(N)}
        &=
        M_{k_0}^{(N)}
        \frac1{2N+1}
        \sum_{j=k_0}^{k-1}a_{j,N}^{-1}+
        \frac1{2N+1}
        \sum_{j=k_0}^{k-1}
        a_{j,N}^{-1}
        \bigl(M_j^{(N)}-M_{k_0}^{(N)}\bigr).
\end{aligned}
\end{equation}
By Lemma~\ref{lem:M-flat-away-origin},
\[
        \sup_{k_0\le j\le m_N}
        |M_j^{(N)}-M_{k_0}^{(N)}|
        \xrightarrow[N\to\infty]{\mathbb P}0.
\]
Moreover, by Lemma~\ref{lem:a-estimates} (Equation \eqref{eq:sum_a-1}),
\[
        \frac1{2N+1}
        \sum_{j=k_0}^{m_N}a_{j,N}^{-1}
        =
        O(N^{-1/2}).
\]
Hence,
\begin{equation}\label{eq-contr-aDm}
       \sup_{k_0\le k\le m_N}\frac1{2N+1}\left|
        \sum_{j=k_0}^{k-1}
        a_{j,N}^{-1}
        \bigl(M_j^{(N)}-M_{k_0}^{(N)}\bigr)\right|
        =
        o_{\mathbb P}(N^{-1/2}).
\end{equation}
Also, by Lemma~\ref{lem:a-estimates} (Equation \eqref{eq:sum_a-1}) again,
\[
        \frac1{2N+1}
        \sum_{j=k_0}^{k-1}a_{j,N}^{-1}
        =
        \frac{2}{\sqrt\pi}\frac{\sqrt k}{N}
        +
        o(N^{-1/2})
\]
uniformly in \(k\). Since the sequence \(\big(M_{k_0}^{(N)}\big)_{N}\) is tight from Proposition \ref{prop:B-seeding}, we obtain from \eqref{eq-decom-aM} and \eqref{eq-contr-aDm} that
\begin{equation}\label{eq-equivalent-a}
        \frac1{2N+1}
        \sum_{j=k_0}^{k-1}a_{j,N}^{-1}M_j^{(N)}
        =
        \frac{2}{\sqrt\pi}\frac{\sqrt k}{N}
        M_{k_0}^{(N)}
        +
        o_{\mathbb P}(N^{-1/2}),
\end{equation}
uniformly for \(k_0\le k\le m_N\).

It remains to handle the initial term $\overline{C}_{k_0}$. By Lemma~\ref{lem:uv-apriori}, there exists a constant $K$ such that
\begin{equation}\label{eq-control-Ck0}
        \mathbb E[\sqrt N\,\overline C_{k_0}]
        \le
        K\sqrt{\frac{k_0}{N}}
        \xrightarrow[N\to\infty]{}0.
\end{equation}
Therefore from Equations \eqref{eq-decomp-C}, \eqref{eq-control-Ck0}, \eqref{eq-equivalent-a} and \eqref{eq-control-R} we get
\begin{equation}\label{eq-control-CBar}
        \sup_{k_0\le k\le m_N}
        \frac{N-k}{\sqrt N}
        \left|
        \overline C_k
        -
        \frac{2}{\sqrt\pi}
        \frac{\sqrt k}{N}
        M_{k_0}^{(N)}
        \right|
        \xrightarrow[N\to\infty]{\mathbb P}0.
\end{equation}
Since \(C_k=(N-k)\overline C_k\), this proves \eqref{eq-equiv-C}.

We now prove \eqref{eq-equiv-B}. Since
\[
        \overline B_k=\overline C_k+X_k,
        \qquad\qquad
        X_k=\frac{M_k^{(N)}}{a_{k,N}} \qquad\qquad\text{and }\qquad B_k=k\overline B_k,
\]
we have
\begin{equation}\label{eq-decomp-B}
        \frac{B_k}{\sqrt N}
        =
        \frac{k\overline C_k}{\sqrt N}
        +
        \frac{k}{\sqrt N}
        \frac{M_k^{(N)}}{a_{k,N}}.
\end{equation}

For the first term, using \(k/(N-k)\le \varepsilon^{-1}\) as long as 
\(k\le m_N\le (1-\epsilon)N\), we get
\begin{equation}\label{eq-equiv-B-part1}
\sup_{k_0\le k\le m_N}
        \left|
        \frac{k\overline C_k}{\sqrt N}
        -
        \frac{2}{\sqrt\pi}y_k^{3/2}M_{k_0}^{(N)}
        \right|\le
        \frac1\varepsilon
        \sup_{k_0\le k\le m_N}
        \frac{N-k}{\sqrt N}
        \left|
        \overline C_k
        -
        \frac{2}{\sqrt\pi}
        \frac{\sqrt k}{N}
        M_{k_0}^{(N)}
        \right|\xrightarrow[N\rightarrow\infty]{\mathbb{P}}0
\end{equation}
from Equation \eqref{eq-control-CBar}.

For the second term, we decompose

\[
\begin{aligned}
\sup_{k_0\le k\le m_N}
&\left|
        \frac{k}{\sqrt N}
        \frac{M_k^{(N)}}{a_{k,N}}
        -
        \frac{2}{\sqrt\pi}\sqrt{y_k}\,M_{k_0}^{(N)}
\right|\\& \le \sup_{k_0\le k\le m_N} \frac{k}{\sqrt{N} a_{k,N}}
\left|M_{k}^{(N)}-M_{k_0}^{(N)}
\right| + \big|M_{k_0}^{(N)}\big|\sup_{k_0\le k\le m_N}
\left|\frac{k}{\sqrt N}
        \frac{1}{a_{k,N}}-\frac{2}{\sqrt\pi}\sqrt{y_k}
\right| 
\end{aligned}
\]

The second term converges to $0$ in probability as $N$ goes to infinity from Lemma \ref{lem:a-estimates} (Equation \eqref{uniform-asymptotics-a}) and since the sequence $(M_{k_0}^{(N)})_N$ is tight from Proposition \ref{prop:B-seeding}. The first term converges to $0$ in probability as $N$ goes to infinity since $(\frac{k}{\sqrt{N}a_{k,N}})$ is uniformly bounded when $k_0\leq k\leq m_N$ and from Lemma \ref{lem:M-flat-away-origin}. Thus
\begin{equation}\label{eq-equiv-B-part2}
        \sup_{k_0\le k\le m_N}
        \left|
        \frac{k}{\sqrt N}
        \frac{M_k^{(N)}}{a_{k,N}}
        -
        \frac{2}{\sqrt\pi}\sqrt{y_k}M_{k_0}^{(N)}
        \right|
        \xrightarrow[N\to\infty]{\mathbb P}0.
\end{equation}
Combining the decomposition \eqref{eq-decomp-B} with the two estimates \eqref{eq-equiv-B-part1} and \eqref{eq-equiv-B-part2} gives \eqref{eq-equiv-B}.
\end{proof}

\section{Gluing}\label{sec-gluing}
\subsection{Initial window}

We first prove the following Lemma.

\begin{lemma}[Positive variation during a stable block]
\label{lem:stable-block-positive-variation}
Fix \(1\le k\le N-1\), and define
\[
        P_k
        :=
        \sum_{j=0}^{L-1}
        \bigl(
        C^{(j+1)}-C^{(j)}
        \bigr)^+.
\]
Then
\[
        \mathbb E\left[
        P_k
        \,\middle|\,
        \mathcal F_{S_k}
        \right]
        \le
        \frac{N-k}{k}\,\overline C_k+\overline B_k.
\]
Consequently, there exists \(K<\infty\), independent of \(N\) and \(k\),
such that
\[
        \mathbb E[P_k]
        \le
        \frac{K}{\sqrt k}.
\]
\end{lemma}

\begin{proof}
Recall that \(d:=N-k\), \(q:=d/N\), and $\lambda=\frac{k}{2Nd}$.
During the \((j+1)\)-st stable step,
\[
        \bigl(
        C^{(j+1)}-C^{(j)}
        \bigr)^+
        \le
        \frac12\bigl(w(\mu)+w(\pi)\bigr).
\]
Hence, on \(\{L\ge j+1\}\),
\[
\begin{aligned}
        &\mathbb E\left[
        \bigl(
        C^{(j+1)}-C^{(j)}
        \bigr)^+
        \,\middle|\,
        \mathcal F_{S_k+j},
        L\ge j+1
        \right]                                      \\
        &\qquad\le
        \frac12\left(
        \overline C^{(j)}
        +
        \frac{k}{N}\overline B_k
        +
        \frac{d}{N}\overline C^{(j)}
        \right)                                      \\
        &\qquad\le
        \overline C^{(j)}
        +
        \frac{k}{2N}\overline B_k.
\end{aligned}
\]
Let
\[
        m_j
        :=
        \mathbb E\left[
        \overline C^{(j)}
        \,\middle|\,
        \mathcal F_{S_k},\,L\ge j
        \right].
\]
By Lemma~\ref{lem:stable-block-average},
\[
        m_j
        =
        \overline B_k
        +
        (1-\lambda)^j
        \bigl(
        \overline C_k-\overline B_k
        \bigr).
\]
Since
\[
        \mathbb P(L\ge j+1\mid\mathcal F_{S_k})=q^{j+1},
\]
we obtain
\[
\begin{aligned}
        \mathbb E\left[
        P_k
        \,\middle|\,
        \mathcal F_{S_k}
        \right]
        &\le
        \sum_{j\ge0}q^{j+1}m_j
        +
        \frac{k}{2N}\frac{q}{1-q}\overline B_k.
\end{aligned}
\]
Moreover,
\[
\begin{aligned}
        \sum_{j\ge0}q^{j+1}m_j
        &=
        \frac{q}{1-q(1-\lambda)}\overline C_k
        +
        \left(
        \frac{q}{1-q}
        -
        \frac{q}{1-q(1-\lambda)}
        \right)\overline B_k.
\end{aligned}
\]
Using that
\[
        \frac{q}{1-q(1-\lambda)}
        =
        \frac{2N(N-k)}{k(2N+1)},
\]
\[
        \frac{q}{1-q}
        -
        \frac{q}{1-q(1-\lambda)}
        =
        \frac{N-k}{k(2N+1)},
\]
and
\[
        \frac{k}{2N}\frac{q}{1-q}
        =
        \frac{N-k}{2N},
\]
we conclude that
\[
        \mathbb E\left[
        P_k
        \,\middle|\,
        \mathcal F_{S_k}
        \right]
        \le
        \frac{N-k}{k}\,\overline C_k+\overline B_k.
\]

Finally, Lemma~\ref{lem:uv-apriori} gives
\[
        \mathbb E\left[
        \frac{N-k}{k}\overline C_k
        \right]
        =
        \frac{\mathbb E[C_k]}{k}
        \le
        \frac{K}{\sqrt k},
\]
and
\[
        \mathbb E[\overline B_k]
        \le
        \mathbb E[\overline B_k+\overline C_k]
        \le
        \frac{K}{\sqrt k}.
\]
This proves the second assertion.
\end{proof}

Lemma \ref{lem:stable-block-positive-variation} enables the control of variations of the stochastic processes $B$ and $C$ during the initial window. This small estimate near the origin will be used to glue the mesoscopic entrance regime with the macroscopic one.

\begin{lemma}[Initial window]
\label{lem:initial-window}
Let \(k_0=k_0(N)\) be an integer sequence such that
\[
        \frac{k_0}{N}\longrightarrow0.
\]
Then
\[
        \sup_{0\le y\le k_0/N}
        \left|
        \frac1{\sqrt N}
        \binom{B_{k_N(y)}}{C_{k_N(y)}}
        \right|
        \xrightarrow[N\to\infty]{\mathbb P}0.
\]
\end{lemma}

\begin{proof}
Since \((B_k)_{1\leq k\leq N}\) is nondecreasing,
\[
        \sup_{1\le k\le k_0}B_k=B_{k_0}.
\]
By Lemma~\ref{lem:uv-apriori},
\[
        \mathbb E[B_{k_0}]
        \le
        K\sqrt{k_0}.
\]

For \(C\), let \(P_k\) be the positive variation accumulated during the
stable part of the block \([S_k,S_{k+1})\), as in
Lemma~\ref{lem:stable-block-positive-variation}.
Since the final jump of each block can only decrease \(C\) and \(C_1=0\),
\[
        \sup_{1\le k\le k_0}C_{k}
        \le
        \sum_{k=1}^{k_0-1}P_k.
\]
Hence, by Lemma~\ref{lem:stable-block-positive-variation},
\[
        \mathbb E\left[
        \sup_{1\le k\le k_0}C_k
        \right]
        \le
        K\sum_{k=1}^{k_0-1}\frac1{\sqrt k}
        \le
        K\sqrt{k_0}.
\]
Therefore
\[
        \frac1{\sqrt N}
        \mathbb E\left[
        \sup_{1\le k\le k_0}(B_k+C_k)
        \right]
        \le
        K\sqrt{\frac{k_0}{N}}
        \longrightarrow0.
\]
The conclusion follows from Markov's inequality.
\end{proof}

\subsection{Terminal window and continuity at \texorpdfstring{\(y=1\)}{y=1}}

We close the proof at the endpoint without using any martingale estimate.
We retain the notation introduced above and set \(C_N:=0\). Lemma \ref{lem:stable-block-positive-variation} enables the control of variations of the stochastic processes $B$ and $C$ during the terminal window.

\begin{lemma}[Small variation in the terminal window]
\label{lem:terminal-window}
For \(\varepsilon\in(0,1/2)\), set
\begin{equation*}
        k_\varepsilon:=\lfloor(1-\varepsilon)N\rfloor.
\end{equation*}
There exists \(K<\infty\), independent of \(N\) and \(\varepsilon\),
such that, for every fixed \(\varepsilon\in(0,1/2)\) and all sufficiently large \(N\),
\begin{equation}\label{eq-terminal-B}
        \mathbb E[B_N-B_{k_\varepsilon}]
        \le
        K\varepsilon\sqrt N
\end{equation}
and
\begin{equation}\label{eq-terminal-C}
        \mathbb E\left[
        \sup_{k_\varepsilon\le k\le N}C_k
        \right]
        \le
        K\varepsilon\sqrt N.
\end{equation}
Consequently, for every \(\eta>0\), and after changing the constant $K$,
\begin{equation}\label{eq-terminal-BandC}
        \limsup_{N\to\infty}
        \mathbb P\left(
        \sup_{k_\varepsilon\le k\le N}
        \frac{|B_k-B_{k_\varepsilon}|+C_k}{\sqrt N}
        >\eta
        \right)
        \le
        \frac{K\varepsilon}{\eta}.
\end{equation}
\end{lemma}

\begin{proof}
Fix \(\varepsilon\in(0,1/2)\). For all sufficiently large \(N\),
\begin{equation}
\label{eq:terminal-window-rounding}
        N-k_\varepsilon
        \le
        \varepsilon N+1
        \le
        2\varepsilon N,
        \qquad
        k_\varepsilon\ge\frac N3.
\end{equation}

We first control the process \((B_k)_{k_{\epsilon}\leq k\leq N}\). As mentioned at the beginning of the proof of Lemma \ref{lem:initial-window}, the process $(B_k)_{1\leq k \leq N}$ is nondecreasing:
stable steps do not modify it, while at the jump \(k\to k+1\),
\begin{equation*}
        B_{k+1}-B_k
        =
        \frac12\bigl(w(\mu_{T_k})+w(\pi_{T_k})\bigr),
\end{equation*} where \(T_k:=S_{k+1}-1\).
Conditionally on the pre-jump configuration, the mother is uniform on the advantaged class, whereas the father is uniform on the whole population. Hence
\begin{equation*}
        \mathbb E[B_{k+1}-B_k]
        \le
        \frac12\mathbb E[\overline B_k]
        +
        \frac1{2N}\mathbb E[B_k+C_{T_k}],
        \qquad 1\le k\le N-1,
\end{equation*}
For \(k\le N-2\), Lemma~\ref{lem:C-before-jump} gives
\begin{equation*}
        \mathbb E[C_{T_k}]
        \le
        K\sqrt{k+1},
\end{equation*}
while, for \(k=N-1\), the trivial bound \(C_{T_k}\le1\) suffices. Together with Lemma~\ref{lem:uv-apriori} and
\(k\ge k_\varepsilon\ge N/3\), this yields
\begin{equation*}
        \mathbb E[B_{k+1}-B_k]
        \le
        \frac{K}{\sqrt N},
        \qquad
        k_\varepsilon\le k\le N-1.
\end{equation*}
Summing and using \eqref{eq:terminal-window-rounding}, we obtain \eqref{eq-terminal-B}.

We next control the process \((C_k)_{k_{\epsilon}\leq k\leq N}\). Let \(P_k\) denote the positive variation accumulated
during the stable part of the block \([S_k,S_{k+1})\), as in
Lemma~\ref{lem:stable-block-positive-variation}.
Since the final jump of each block can only decrease \(C\), for every
\(m\ge k_\varepsilon\),
\[
        C_m
        \le
        C_{k_\varepsilon}
        +
        \sum_{k=k_\varepsilon}^{m-1}P_k.
\]
Therefore, by Lemmas~\ref{lem:uv-apriori} and
\ref{lem:stable-block-positive-variation},
\[
\begin{aligned}
        \mathbb E\left[
        \sup_{k_\varepsilon\le m\le N}C_m
        \right]
        &\le
        \mathbb E[C_{k_\varepsilon}]
        +
        \sum_{k=k_\varepsilon}^{N-1}\mathbb E[P_k] \\
        &\le
        K\varepsilon\sqrt N
        +
        K\sum_{k=k_\varepsilon}^{N-1}\frac1{\sqrt k}.
\end{aligned}
\]
Using \eqref{eq:terminal-window-rounding},
\[
        \sum_{k=k_\varepsilon}^{N-1}\frac1{\sqrt k}
        \le
        \frac{N-k_\varepsilon}{\sqrt{k_\varepsilon}}
        \le
        K\varepsilon\sqrt N.
\]
This proves \eqref{eq-terminal-C}.

Finally, by monotonicity of \(B\),
\begin{equation*}
        \sup_{k_\varepsilon\le k\le N}
        |B_k-B_{k_\varepsilon}|
        =
        B_N-B_{k_\varepsilon}.
\end{equation*}
Markov's inequality, together with
\eqref{eq-terminal-B} and
\eqref{eq-terminal-C}, gives, for all sufficiently large \(N\),
\begin{equation*}
        \mathbb P\left(
        \sup_{k_\varepsilon\le k\le N}
        \frac{|B_k-B_{k_\varepsilon}|+C_k}{\sqrt N}
        >\eta
        \right)
        \le
        \frac{K\varepsilon}{\eta}.
\end{equation*}
Taking the upper limit as \(N\to\infty\) proves \eqref{eq-terminal-BandC}.
\end{proof}

\subsection{Proof of the Theorem}

\begin{proof}[Proof of Theorem~\ref{thm:main}]
Set
\[
        f(y):=
        \binom{b(y)}{c(y)},
        \qquad
        Z_N(y):=
        \frac1{\sqrt N}
        \binom{B_{k_N(y)}}{C_{k_N(y)}},
        \qquad 0\le y\le1.
\]
Let \(k_0=k_0(N)\) be such that
\[
        1\ll k_0\ll N,
\]
and set
\[
        A_N:=M_{k_0}^{(N)}.
\]
By Proposition~\ref{prop:B-seeding},
\begin{equation}\label{eq:AN-to-W}
        A_N\Longrightarrow W.
\end{equation}
In particular, \((A_N)_{N\ge1}\) is tight.

Fix \(\varepsilon\in(0,1/2)\). By
Proposition~\ref{prop:deterministic-propagation},
\[
        \sup_{k_0\le k\le\lfloor(1-\varepsilon)N\rfloor}
        \left|
        \frac1{\sqrt N}
        \binom{B_k}{C_k}
        -
        A_N f\left(\frac{k}{N}\right)
        \right|
        \xrightarrow[N\to\infty]{\mathbb P}0.
\]
Since \(f\) is uniformly continuous and \(A_N\) is tight,
\[
        \sup_{k_0/N\le y\le1-\varepsilon}
        \left|
        Z_N(y)-A_Nf(y)
        \right|
        \xrightarrow[N\to\infty]{\mathbb P}0.
\]
Moreover, Lemma~\ref{lem:initial-window}, the continuity of $f$, together with the fact that
\(f(0)=0\), \(k_0/N\to0\), and the tightness of \(A_N\), give
\[
        \sup_{0\le y\le k_0/N}
        \left|
        Z_N(y)-A_Nf(y)
        \right|
        \xrightarrow[N\to\infty]{\mathbb P}0.
\]
Consequently,
\begin{equation}\label{eq:main-away-terminal}
        D_N(\varepsilon)
        :=
        \sup_{0\le y\le1-\varepsilon}
        \left|
        Z_N(y)-A_Nf(y)
        \right|
        \xrightarrow[N\to\infty]{\mathbb P}0.
\end{equation}

It remains to control the terminal interval. Set
\[
        k_\varepsilon:=\lfloor(1-\varepsilon)N\rfloor
\]
and
\[
        T_N(\varepsilon)
        :=
        \sup_{1-\varepsilon\le y\le1}
        \left|
        Z_N(y)-Z_N(1-\varepsilon)
        \right|.
\]
Since \(B_k\) is nondecreasing and \(C_k\ge0\),
\[
        T_N(\varepsilon)
        \le
        \frac{B_N-B_{k_\varepsilon}}{\sqrt N}
        +
        \frac{1}{\sqrt N}
        \sup_{k_\varepsilon\le k\le N}C_k.
\]
Therefore, by Lemma~\ref{lem:terminal-window} (Equations \eqref{eq-terminal-B} and \eqref{eq-terminal-C}), for all sufficiently large \(N\),
\begin{equation}\label{eq:terminal-path-control}
        \mathbb E[T_N(\varepsilon)]
        \le
        K\varepsilon.
\end{equation}

Define
\[
        \omega_f(\varepsilon)
        :=
        \sup_{1-\varepsilon\le y\le1}
        |f(y)-f(1-\varepsilon)|.
\]
Since \(f\) is continuous on \([0,1]\),
\[
        \omega_f(\varepsilon)\longrightarrow0
        \qquad\text{as }\varepsilon\downarrow0.
\]
For \(1-\varepsilon\le y\le1\),
\[
        |Z_N(y)-A_Nf(y)|
        \le
        |Z_N(y)-Z_N(1-\varepsilon)| \quad+
        |Z_N(1-\varepsilon)-A_Nf(1-\varepsilon)| \quad+
        |A_N|\,|f(y)-f(1-\varepsilon)|.
\]
Together with \eqref{eq:main-away-terminal}, this gives
\begin{equation}\label{eq:global-profile-bound}
        \|Z_N-A_Nf\|_\infty
        \le T_N(\varepsilon)+
        D_N(\varepsilon)
        +|A_N|\,\omega_f(\varepsilon).
\end{equation}

Let \(\eta>0\). By \eqref{eq:main-away-terminal},
\eqref{eq:terminal-path-control}, \eqref{eq:global-profile-bound} and Markov's inequality,
\[
\begin{aligned}
        \limsup_{N\to\infty}
        \mathbb P\left(
        \|Z_N-A_Nf\|_\infty>3\eta
        \right)
        &\le
        \frac{K\varepsilon}{\eta}
        +
        \limsup_{N\to\infty}
        \mathbb P\left(
        |A_N|\,\omega_f(\varepsilon)>\eta
        \right).
\end{aligned}
\]
Since \((A_N)\) is tight and
\(\omega_f(\varepsilon)\to0\),
the second term on the right tends to \(0\) as
\(\varepsilon\downarrow0\). Letting \(\varepsilon\downarrow0\) therefore yields
\begin{equation}
        \|Z_N-A_Nf\|_\infty
        \xrightarrow[N\to\infty]{\mathbb P}0.
\end{equation}

Now the map
\[
        r\longmapsto rf
\]
is continuous from \(\mathbb R\) into
\(D([0,1],\mathbb R^2)\) endowed with the Skorokhod \(J_1\) topology.
Hence, by \eqref{eq:AN-to-W} and the continuous mapping theorem,
\[
        A_Nf\Longrightarrow Wf.
\]
Moreover,
\[
        d_{J_1}(Z_N,A_Nf)
        \le
        \|Z_N-A_Nf\|_\infty
        \xrightarrow[N\to\infty]{\mathbb P}0.
\]
Slutsky's lemma therefore gives
\[
        Z_N\Longrightarrow Wf
\]
in \(D([0,1],\mathbb R^2)\), endowed with the Skorokhod \(J_1\) topology.

The characterization of \(W\), its uniqueness among \(L^2\) solutions with
mean \(1\), and its strict positivity follow from
Proposition~\ref{prop:RRT-smoothing}.
\end{proof}

\paragraph{AI use declaration} 

This work was initiated in early 2025. By the summer, we had identified a roadmap for the proof: the early phase described by a RRT, followed by a regime in which the martingales plateau, together with a strategy for gluing the two phases. Many technical details, however, remained unchecked. AI (ChatGPT 5 and its subsequent versions) was subsequently used to help us derive the numerous lemmas needed to turn this roadmap into a complete argument. In particular, it was especially helpful in establishing the estimates surrounding the quadratic variation in the second phase (Section 6). The authors carefully verified all proofs and rewrote them where necessary. They take full responsibility for the final text and all mathematical results.

\bibliographystyle{plain}
\bibliography{biblio}

\end{document}